\RequirePackage{ifthen}
\newboolean{mpc}
\setboolean{mpc}{false}

\ifthenelse {\boolean{mpc}}
{
\documentclass[smallextended]{svjour3}
\smartqed  
} {
\documentclass[10pt]{article}
\usepackage[margin=1.35in]{geometry}
}

\usepackage{graphicx}
\usepackage{amsmath}
\usepackage{amssymb}
\usepackage{bm}
\usepackage{hyperref}
\usepackage{color}
\usepackage{multirow}
\usepackage{mathrsfs} 
\usepackage{soul} 
\usepackage{stmaryrd} 
\usepackage[makeroom]{cancel} 
\usepackage{braket,amsfonts}
\usepackage{array}
\usepackage{pgfplots}
\usepackage{algorithmic}
\usepackage{epstopdf}
\usepackage{amsopn}
\usepackage{mathtools}
\usepackage{comment}
\usepackage{tikz}
\usepackage{accents}
\usepackage{longtable}
\usepackage[thinc]{esdiff}
\usepackage{pdflscape}
\usepackage{hyperref}
\usepackage{caption}
\usepackage{subcaption}
\usepackage{lineno}
\usepackage{comment}
\usepackage{accents}
\usepackage{dsfont}
\usepackage{multirow}
\usepackage{booktabs}
\usepackage[linesnumbered,ruled,vlined]{algorithm2e}
\newcommand{\fnm}[1]{#1}
\newcommand{\sur}[1]{#1}

\makeatletter
\let\cl@chapter\undefined
\makeatletter

\usepackage[capitalize,noabbrev]{cleveref}

\ifthenelse {\boolean{mpc}}
{

} {
\usepackage[round]{natbib}
\usepackage{amsthm}
\newtheorem{theorem}{Theorem}
\newtheorem{proposition}{Proposition}
\newtheorem{lemma}{Lemma}

\newtheorem{remark}{Remark}

\newtheorem{definition}{Definition}

\newtheorem{example}{Example}
\Crefname{section}{Sect.}{Sects.}
\Crefname{proposition}{Prop.}{Props.}
\Crefname{theorem}{Thm.}{Thms.}
\Crefname{definition}{Defn.}{Defns.}
\Crefname{corollary}{Cor.}{Cors.}
\Crefname{figure}{Fig.}{Figs.}
\Crefname{observation}{Obs}{Obss.}
\Crefname{chapter}{Chap.}{Chaps.}
\Crefname{table}{Tab.}{Tabs.}
}

\newcommand{\cP}{\mathcal P}

\newcommand{\cE}{\mathcal E}
\newcommand{\cY}{\mathcal Y}
\newcommand{\R}{\mathbb R}
\newcommand{\F}{\mathbb F}
\newcommand{\C}{\mathbb C}

\newcommand{\cZ}{\mathcal Z}

\newcommand{\fS}{\mathfrak S}
\newcommand{\fZ}{\mathfrak Z}
\newcommand{\cF}{\mathcal F}

\newcommand{\cC}{\mathcal C}

\newcommand{\cM}{\mathcal M}

\newcommand{\orgname}[1]{#1}%
\newcommand{\orgaddress}[1]{#1}%
\newcommand{\postcode}[1]{#1}%
\newcommand{\city}[1]{#1}%
\newcommand{\country}[1]{#1}%

\DeclareMathOperator{\conv}{conv}

\DeclareMathOperator{\sep}{Sep}

\DeclareMathOperator{\dmat}{\mathbb{D}}

\DeclareMathOperator{\SOST}{PSDT}

\newcommand{\X}{\mathbf{X}}
\newcommand{\Z}{\mathbf{Z}}

\renewcommand{\X}{\mathbf{X}}
\newcommand{\x}{\mathbf{x}}
\newcommand{\y}{\mathbf{y}}
\newcommand{\p}{\mathbf{p}}

\newcommand{\z}{\mathbf{z}}
\renewcommand{\u}{\mathbf{u}}
\renewcommand{\v}{\mathbf{v}}
\newcommand{\U}{\mathbf{U}}
\newcommand{\V}{\mathbf{V}}
\newcommand{\bc}{\mathbf{c}}

\newcommand{\Id}{\mathbf{Id}}

\newcommand{\vi}{\mathbf{i}}

\renewcommand{\a}{\mathbf{a}}

\newcommand{\opt}{\text{opt}}
\newcommand{\heur}{\text{heur}}
\newcommand{\relx}{\text{relax}}

\newcommand{\lb}{\text{lb}}
\newcommand{\ub}{\text{ub}}
\newcommand{\feas}{\text{feas}}

\newcommand{\NPh}{\textsf{NP}-hard}
\DeclareMathOperator{\herm}{\mathbb{S}}
\DeclareMathOperator{\ext}{ext}
\DeclareMathOperator{\tang}{T}
\DeclareMathOperator{\argmin}{argmin}
\DeclareMathOperator{\dps}{DPS}

\DeclareMathOperator{\tr}{Tr}

\newcommand{\Tr}{\mathsf{T}}

\DeclareMathOperator{\mtens}{\otimes}
\DeclareMathOperator{\bmtens}{\bigotimes}
\newcommand{\inner}[2]{\langle #1, #2 \rangle}

\newcommand{\norm}[1]{\lVert#1\rVert}
\newcommand{\vc}[1]{\mathbf{#1}}
\newcommand{\vd}{\vc{d}}

\newcommand{\deq}{\coloneqq}

\newcommand{\eg}{e.g.,~}

\usepackage{ulem}

\newcommand{\kywrds}{Tensor Optimization \and Cutting-Plane Methods \and Branch-and-Bound Algorithms \and Quantum Entanglement \and Convex Relaxations}

\title{Convex semidefinite tensor optimization and quantum entanglement}

\ifthenelse {\boolean{mpc}}
{
\titlerunning{Convex semidefinite tensor optimization and quantum entanglement}
\authorrunning{Xu et al.}

\author{\fnm{Liding} \sur{Xu} \and \fnm{Ye-Chao} \sur{Liu} \and \fnm{Sebastian} \sur{Pokutta}}

\institute{
   \fnm{Liding} \sur{Xu},  \fnm{Ye-Chao} \sur{Liu}, \fnm{Sebastian} \sur{Pokutta} \at
 \orgname{Zuse Institute Berlin}, \orgaddress{\city{Berlin}, \postcode{14195}, \country{Germany}} \\
 \email{lidingxu.ac@gmail.com, yechaoliu1994@outlook.com, pokutta@zib.de}
}

}
{
\author{\fnm{Liding} \sur{Xu} \and \fnm{Ye-Chao} \sur{Liu} \and \fnm{Sebastian} \sur{Pokutta}
 \thanks{
    \orgname{Zuse Institute Berlin, Berlin, Germany.
E-mail: {lidingxu.ac@gmail.com, yechaoliu1994@outlook.com, pokutta@zib.de}}
 }}
}

\begin{document}

\maketitle


\begin{abstract}
 The cone of positive-semidefinite (PSD) matrices is fundamental in convex optimization, and we extend this notion to tensors, defining PSD tensors, which correspond to separable  quantum states. We study the convex optimization problem over the PSD tensor cone. While this convex cone admits a smooth reparameterization through tensor factorizations (analogous to the matrix case), it is not self-dual. Moreover, there are currently no efficient algorithms for projecting onto or testing membership in this cone, and the semidefinite tensor optimization problem, although convex, is NP-hard. To address these challenges, we develop methods for computing lower and upper bounds on the optimal value of the problem. We propose a general-purpose iterative refinement algorithm that combines a lifted alternating direction method of multipliers with a cutting-plane approach. This algorithm exploits PSD tensor factorizations to produce heuristic solutions and refine the solutions using cutting planes. Since the method requires a linear minimization oracle over PSD tensors, we design a spatial branch-and-bound algorithm based on convex relaxations and valid inequalities. Our framework allows us to study the white-noise mixing threshold, which characterizes the entanglement properties of quantum states. Numerical experiments on benchmark instances demonstrate the effectiveness of the proposed methods.
\ifthenelse {\boolean{mpc}}
{
}{
}
\end{abstract}

\section{Introduction}

In quantum information science, a (quantum) state is represented by a density matrix, which is a Hermitian, positive-semidefinite (PSD) matrix with unit trace. The density matrix of a pure state has rank one. A pure state of a multipartite system is separable if and only if its density matrix can be written as a tensor product of rank-one density matrices corresponding to pure states of the individual subsystems.  A mixed (non-pure) state is separable if its density matrix is a convex combination of density matrices of separable pure states. States that are not separable are called entangled. Detecting entanglement in mixed states is hard; however, it is central to quantum information science and impacts many areas of physics \cite{Eisert2010colloqium,Nishioka2018Entanlgement}.

Mathematically, separability of (quantum) states is closely related to decompositions of tensors. Canonical polyadic decomposition (CPD) is the most well-known tensor decomposition. Given a vector \( \vd = (d_1, \dots, d_m) \), an order-$m$ tensor $\y \in \F^{d_1 \times \cdots \times d_m}$ admits  a CPD, if $\y =  \sum_{i \in [r]} \lambda_i \x_{i1} \mtens \cdots \mtens \x_{im}$,
where $\F$ is the field $\R$ of real numbers or the field $\C$ of complex numbers. The minimum $r$ in CPD of $\y$ is its CPD rank.
If all \textit{modes} $\x_{ij}$ ($j \in [m]$) in each \textit{component} $\x_{i1} \mtens  \cdots  \mtens \x_{im}$  are identical to a vector $\x_{i}$ and all $d_j = d$, then CPD is symmetric (under permutation of indices).

Separable states are tensors that have more restrictive decompositions.  We define the following cone of (order-$2m$) PSD  tensors:
\begin{equation}
    \label{eq.psdt}
    \SOST_{\F}(\vd) \deq \{ \y \in \F^{d_1 \times d_1 \cdots \times d_m \times d_m} \mid \y = \sum_{i \in [r]}  \x_{i1} \mtens \x_{i1}^\dagger \cdots \x_{im} \mtens \x_{im}^\dagger  \},
\end{equation}
where $\x^\dagger$ denotes the conjugate transpose of $\x$.
Any tensor in $\SOST_{\F}(\vd)$, when reshaped as an order-2 tensor, yields a large PSD matrix of size $\bar{d}^2$ with $\bar{d} = \prod_{j \in [m]}d_j$. The set  $\sep(\vd)$ of separable states consists of  all   unit-trace reshaped  tensors in $\SOST_{\C}(\vd)$:
\begin{equation}
    \label{eq.sepstate}
    \sep(\vd) \deq \{ \y \in \SOST_{\C}(\vd)  \mid \tr(\y) = 1 \}.
\end{equation}

Identification of separable states is difficult. For \textit{unipartite} systems ($m=1$), the set of separable states is simply the set of density matrices. For \textit{multipartite systems}, however, separable states form a strict subset of density matrices. In addition, the set \( \sep(\vd) \) is \textit{computationally intractable}: there are no known efficient algorithms for separation, projection, linear minimization, or even membership testing. The membership testing problem is already strongly \NPh{} even for bipartite systems ($m=2$)~\cite{gharibian2008strong}. The linear minimization problem $\min_{\y \in \sep(\vd)} \inner{\chi}{\y}$, known as the \textit{best separable state (BSS)} problem, is also \NPh{}~\cite{fang2021sum}.

We consider a generalization of semidefinite programming (SDP): \textit{semidefinite tensor optimization (SDTO)}, a conic optimization problem over the cone of PSD tensors. The SDTO can be written as:
\begin{equation}
    \label{eq:P}\tag{SDTO}
   \opt \deq \min \{\inner{\bc}{\z}: \z \in \cZ,\;  A(\z) + \a \in  \SOST_{\F}(\vd)\},
\end{equation}
where $\cZ$ is a tractable convex cone, $A$ is a linear map, $\a$ is a constant tensor. The intractability of $\sep(\vd)$ implies that SDTO is also \NPh{}. When $m=1$ and $\cZ$ is a PSD cone, \eqref{eq:P} becomes an SDP, for which off-the-shelf solvers are available. However, algorithms for tackling the general problem \eqref{eq:P} are scarce. We adapt established SDP strategies and address the nonconvexity introduced by PSD tensor factorizations.

Our first contribution is an iterative refinement (IR) algorithm for \eqref{eq:P}. IR alternates between a heuristic stage -- where a lifted alternating direction method of multipliers (LADMM) solves a factorized nonconvex reformulation to find heuristic solutions -- and a refinement stage that tightens and certifies bounds via a cutting-plane (CP) algorithm. LADMM and CP warm-start one another: LADMM outputs accelerate CP, while CP polishes the accuracy of heuristic solutions and provides numerically reliable bounds on $\opt$. The IR algorithm produces a non-increasing sequence of upper bounds, and it can also provide lower bounds and certified optima.

Our second contribution is a spatial branch-and-bound (sBB) algorithm that serves as a linear minimization oracle (LMO) inside the CP algorithm. The convex relaxations in the sBB algorithm are constructed from the intersections of the following building blocks:
\begin{itemize}
    \item Factorizable DPS constraints obtained by combining the Doherty--Parrilo--Spedalieri (DPS) convex outer approximations of separable bipartite states   \cite{doherty2004complete} with McCormick factorization for multilinear monomials~\cite{al1983jointly};
    \item SDP McCormick inequalities generated via Kronecker product of two matrix constraints  ~\cite{anstreicher2017kronecker};
    \item LP McCormick inequalities that support efficient branching rules.
\end{itemize}
We show that these convex sets and inequalities can be derived from a \textit{unified framework} that generalizes the classical \textit{Reformulation--Linearization Technique (RLT)}~\cite{sherali1992global} for polynomial optimization. Moreover, the convex outer approximations of separable states  also give a convex relaxation of SDTO.

We evaluate the proposed algorithms and existing methods   on a benchmark problem: determining the white-noise mixing threshold  of an entangled state—that is, the minimal amount of white-noise admixture required to make the state separable, or equivalently, its random robustness \cite{Vidal1999robustness} under a complete depolarizing channel:
\begin{equation}
    \label{eq.thre} \tag{Threshold}
    S(\phi\|\Id) =  \min \{\z \ge 0: \rho(\z; \phi,\Id) \deq (\Id / \bar{d}-\phi)\z +  \phi \in \sep(\vd)\},
\end{equation}
where the normalized  identity matrix $\Id/ \bar{d}$ represents the  covariance of the white noise. The interpolated state $\rho(\z; \phi, \Id)$ lies on the line segment joining $\Id/\bar{d}$ and $\phi$. The threshold value indicates the separability/entanglement of $\phi$.

\subsection{Related work}
The problem of certifying entanglement or separability has been studied through several complementary approaches. A well-known example is the positive partial transpose (PPT) criterion~\cite{Horodecki1997separability,peres1996separability}, which is necessary and sufficient for bipartite systems of $\vd=(2,2)$ or $(2,3)$, but only necessary in higher dimensions.
The DPS hierarchy \cite{doherty2004complete,Navascues2009complete} extends the PPT criterion and provides nested convex outer approximations of separable states for bipartite systems. From the dual perspective, the DPS hierarchy corresponds to a sum-of-squares hierarchy for a homogeneous Hermitian polynomial over spheres \cite{fang2021sum}. Although the DPS constraints admit physical interpretations, the hierarchy lacks finite-convergence guarantees \cite{fawzi2021set}.  The DPS constraints are primarily used as entanglement certification (YES answers to entanglement). In addition, the DPS  hierarchy can produce lower bounds on white-noise mixing thresholds for bipartite states; see an implementation in Ket.jl~\cite{araujo_2025_15837771}.

The inner approximations of separable states and alternating SDP-based refinement are combined to find  separability certification (NO-answer to entanglement) \cite{Ohst2024certifying} and upper bounds on  white-noise mixing thresholds. As an alternative, the geometric reconstruction (GR) method \cite{Shang2018convex} finds a separability certification for a target state $\phi$  expressed as a convex combination of two separable states $\rho_c$ and $\rho_x$. In the trial step, the GR method performs a line search to generate an extrapolated state~$\rho(\z; \phi, \Id)$ with $\z < 0$. It then applies the conditional-gradient (CG) method (a.k.a. the Frank-Wolfe method) \cite{fwbook,Frank1956algorithm,pokutta2024frank} to minimize the distance from the extrapolated state to the separable set $\sep(\vd)$. The pure separable state returned by the last call of the LMO in this CG procedure becomes $\rho_c$. In the reconstruction step, the method forms a parallelogram with vertices $\rho_x$, $\Id/\bar{d}$, $\rho_c$, and the extrapolated state $\rho(\z; \phi, \Id)$, such that the target state $\phi$ lies on the line segment joining $\rho_c$ and $\rho_x$ \cite{Hu_2023_algorithm}. Finally, the method verifies $\rho_x$ inside an inner approximation ball of $\sep(\vd)$, which implies the separability of $\phi$.

Recently, a primal-dual GR method has been proposed \cite{liu2025unified} to construct upper and lower bounds on the white-noise mixing threshold. This method first applies the CG algorithm to obtain an upper approximation of the distance from the target state~$\phi$ to the separable set $\sep(\vd)$. It then uses the gap estimation from \cite{thuerck2023learning}, together with a simplicial-partition-based LMO equipped with error bounds, to simultaneously find a lower bound on this distance. Whenever the lower bound is strictly positive, the CG-based entanglement detection yields a entanglement certification for the state $\phi$.

Positive-semidefinite and positive-definite tensors were introduced for symmetric tensors in \cite{qi2013symmetric} (via eigenvalue-based definitions) and extended to non-symmetric tensors in \cite{hu2013determinants}. These notions differ from the factorization-based definition of PSD tensors, which is closely related to CPD; see \cite{hitchcock1927expression,hitchcock1928multiple,kolda2009tensor}. Most tensor problems are \NPh{} in the worst case \cite{hillar2013}. It is believed that CPD exhibits a computational-to-statistical gap \cite{bandeira2018}, with efficient algorithms for symmetric CPD expected to exist for random tensors with ranks up to $\mathcal{O}(d^{m/2})$. Notable CPD algorithms include Jennrich's simultaneous diagonalization \cite{de2004computation,leurgans1993decomposition}, alternating least squares \cite{leurgans1993decomposition}, the tensor power method \cite{de2000best}, and optimization-based approaches \cite{sorber2013optimization}. The tensor power method reconstructs components by solving subproblems of the form \(\max\inner{T}{\bmtens_{j\in[m]}\x}\). For symmetric CPD with orthogonal components, local minima of the component subproblem are provably close to global minima \cite{ge2015escaping}.

Manifold optimization \cite{absil2009optimization,boumal2023introduction} is popular for problems with smooth reparameterizations and has been applied to CPD \cite{breiding2018riemannian,dong2022new}.  Reparameterization alone does not guarantee global optimality of stationary points \cite{levin2025effect}. When
$m=1$, the symmetric CPD reduces to the eigenvalue decomposition, which is commonly used to project a matrix onto the PSD cone. In the Burer-Monteiro factorization for SDP \cite{burer2003nonlinear}, the costly projection is avoided by optimizing over low-rank factors, resulting in a nonconvex problem. Under certain equality constraints, all local minima of the factorized formulation are global \cite{burer2005local}; this result can be extended to nonlinear objectives and some structured constraints \cite{journee2010low}. The augmented Lagrangian method (ALM) is a standard tool for constrained optimization \cite{bertsekas2014constrained} and can handle SDP constraints that are not manifold-reparameterizable \cite{dong2022new,wang2023decomposition,wang2025solving}. The alternating direction method of multipliers (ADMM) is widely used for convex problems with partially separable structure \cite{boyd2011distributed}.

\subsection{Notation}
We denote $\herm^d$ as the space of real symmetric (or complex Hermitian) matrices of size $d \times d$, $\herm^d_+$ as the set of PSD matrices of the same size, and $\dmat^d$ as the set of density matrices of the same size. We let $\Id$ denote the identity matrix whose dimension is determined according to the context. For a vector $\vd =(d_1,\dots,d_m)$, we let $d_{[h]}$ denote the product $\prod_{k \in[h]}d_k$, and denote $d_{[m]}$ by $\bar{d}$.

\subsection{Organization of the paper}
In \Cref{sec.approx}, we discuss smooth parameterizations and duality for SDTO and describe methods for computing lower and upper bounds. \Cref{sec.algo} presents algorithms for SDTO, including LADMM for nonconvex lifted problems and the IR algorithm that combines LADMM with the CP method. In \Cref{sec.sbb} we develop an sBB algorithm to serve as the linear-minimization oracle for the CP method, where we detail convex relaxations, valid inequalities, branching rules, and heuristics. \Cref{sec.numerical} reports numerical experiments on white-noise mixing threshold problems, compares our methods with state-of-the-art  approaches in computational physics, and analyzes the computational results.

\section{Reformulations and relaxations of SDTO}
\label{sec.approx}
In this section, we consider the optimization problem \eqref{eq:P}. As discussed in the introduction, $\cZ$ is tractable but $\SOST_{\F}(\vd)$ is not. Unlike classical SDPs, only lower and upper bounds for \eqref{eq:P} are generally computable. We discuss the problem with the following running example.

\begin{example}
    \label{exp.ref1}
    Recall the white-noise mixing threshold problem \eqref{eq.thre}.
    This is an instance of \eqref{eq:P} with $\inner{\bc}{\z} = \z$, $\cZ = \R_+$, $A(\z) = (\Id / \bar{d}-\phi )\z$, $\a = \phi$, and $\SOST_{\F}(\vd) = \SOST_{\C}(\vd)$. Note that $\tr(A(\z) + \a) = \tr(\a)$ is always equal to one, and this implies that $A(\z) + \a \in \sep(\vd)$. We may view $S(\phi\|\Id)$ as a line-search problem over a convex set. In principle, since $\sep(\vd)$ is convex, we could apply a line-search method. In this work, we formulate it as an SDTO problem and solve it via more general-purpose algorithms.
\end{example}

\subsection{Parameterization and duality}

We discuss a smooth parameterization of $\SOST_{\F}(\vd)$. Recall that the factorization size $r$ is the number of components in the tensor decomposition. We call $r$ a \textit{compact} factorization size if every tensor in $\SOST_{\F}(\vd)$ admits a decomposition with at most $r$ components. We have the following loose upper bound on the compact factorization size.
\begin{proposition}
    \label{prop.compact}
    If $\dim\bigl(\SOST_{\F}(\vd)\bigr)=D$, then there exists a compact factorization size $r\le 2D+1$.
    \end{proposition}
    \begin{proof}
        Carathéodory's theorem states that if a point $x$   lies in the convex hull $\conv(S)$ with $S \subset \R^{D}$, then  $x$ is a convex combination of at most $D+1$ extreme points (or generating points) in $S$. Take $S$ as $\SOST_{\F}(\vd)$, and realize the complex tensor as a real tensor, so we can embed $\SOST_{\F}(\vd)$ in $\R^{2D}$. Note that $\SOST_{\F}(\vd)$ is generated from rank-1 tensors. Thus, a tensor in $\SOST_{\F}(\vd)$ can be expressed as a sum of at most $2D+1$ rank-1 tensors.
    \end{proof}

   Definition \eqref{eq.psdt} yields the factorization-based parameterization of PSD tensors. We can construct a smooth surjective map $\Psi: \F^{r\bar{d}} \to \SOST_{\F}(\vd) \subseteq \herm^{\bar{d}}$ such that
   \begin{equation}
    \label{eq.lift}
    \x  = (\x_{11}, \dots, \x_{rm}) \mapsto \Psi(\x) = \sum_{i \in [r]}  \x_{i1} \mtens \x_{i1}^\dagger \cdots \x_{im} \mtens \x_{im}^\dagger,
   \end{equation}
   where $\x_{ij} $ is in the vector space $\F^{d_j}$ equipped with the standard (Hermitian) inner product, and the factorization size $r$ is compact. Recall that a matrix $\y$ is PSD (equivalently, $\y = [\u_1\dots \u_r] [\u_1\dots \u_r]^\top$), if $\v^\dagger \y \v = \inner{\y}{\v\mtens \v^\dagger} = \sum_{i \in [r]} (\u_i^\top \v)^2$ is nonnegative for any vector $\v$. Given  $\y \in \SOST_{\F}(\vd)$, for vectors $\v_{1}, \dots, \v_{m}$, $\inner{\y}{ \v_{1} \mtens \v_{1}^\dagger \cdots \v_{m} \mtens \v_{m}^\dagger}$ equals $\sum_{i \in [r]}(\x_{i1}^\top \v_{1})^2\cdots(\x_{im}^\top \v_{m})^2$, which is non-negative.  This shows that $\SOST_{\F}(\vd)$ is indeed a generalization of PSD matrices.

   The space $\F^{r\bar{d}}$ together with $\Psi$ is a so-called  lift of $\SOST_{\F}(\vd)$ in the context of manifold optimization \cite{levin2025effect}. Thus, any optimization problem involving $\SOST_{\F}(\vd)$ can be transformed into a lifted nonconvex optimization problem in the space $\F^{r\bar{d}}$. We will show that heuristic solutions for the lifted problem can be found by the LADMM in \Cref{sec.ladmm}.

We next consider the duality for \eqref{eq:P}.
Let $A^*$ be the adjoint of the linear map $A$, let $\cZ^*$ be the dual cone of $\cZ$, and let $\SOST^*_{\F}(\vd)$ be the dual cone of $\SOST_{\F}(\vd)$. Assume a standard regularity condition (e.g., Slater's condition) holds so that there is no duality gap between \eqref{eq:P} and its dual:
    \begin{equation}
   \label{eq:D} \tag{DSDTO}
\opt = \max\{\langle -\a,\y^*\rangle \mid \y^* \in \SOST^*_{\F}(\vd),\; -A^*(\y^*)+\bc\in\cZ^*\}.
\end{equation}

Unlike PSD cones, $\SOST_{\F}(\vd)$ is not self-dual.
\begin{proposition}
    The extreme rays of $\SOST_{\F}(\vd)$ are precisely the rank-one product tensors
    \[\{\y\mid\y=\x_1\mtens\x_1^\dagger\cdots\x_m\mtens\x_m^\dagger\}.\]
    The dual cone can be written as
    \[\SOST^*_{\F}(\vd)=\{\y^*\mid\langle\y^*,\y\rangle\ge0\ \forall\ \y\in\ext(\SOST_{\F}(\vd))\},\]
    and for $m\ge2$ this dual cone strictly contains $\SOST_{\F}(\vd)$.
\end{proposition}

\begin{proof}
  We first show that every tensor of the form
  $
    \y = \x_{1} \mtens \x_{1}^\dagger \cdots \x_{m} \mtens \x_{m}^\dagger
$
  generates an extreme ray of $\SOST_{\F}(\vd)$.
  Suppose $\y$ can be decomposed as $\y = \y_1 + \y_2$ with $\y_1, \y_2 \in \SOST_{\F}(\vd)$.
  Since $\y$ is a rank-one PSD product tensor, any decomposition into PSD product terms must be proportional. Hence there exist $\alpha, \beta \ge 0$ such that $\y_1 = \alpha \y$ and $\y_2 = \beta \y$.
  This proves that $\y$ spans an extreme ray of $\SOST_{\F}(\vd)$. Conversely, let $\y \in \SOST_{\F}(\vd)$ be an extreme ray.
  By definition, $\y$ is a conic combination of elements of the form
  $\x_{1} \mtens \x_{1}^\dagger \cdots \x_{m} \mtens \x_{m}^\dagger$.
  If more than one such term appears with a positive coefficient, then $\y$ can be expressed as the sum of two non-collinear elements of the cone, contradicting extremality.
  Hence $\y$ must itself be of the form above.
  This establishes
$
    \ext(\SOST_{\F}(\vd)) \;=\; \Bigl\{\y \;\big|\; \y = \x_{1} \mtens \x_{1}^\dagger \cdots \x_{m} \mtens \x_{m}^\dagger \Bigr\}.
$
Next, consider the dual cone.
  By definition, $
    \SOST^*_{\F}(\vd) = \{\y^* \;\mid\; \inner{\y^\ast}{\y} \ge 0, \;\forall \y \in \SOST_{\F}(\vd)\}.
  $
  Since a closed convex cone is generated by its extreme rays, it suffices to test positivity against $\ext(\SOST_{\F}(\vd))$.
  Therefore,
  \[
    \SOST^*_{\F}(\vd) = \{\y^* \;\mid\; \inner{\y^\ast}{\y} \ge 0, \;\forall \y \in \ext(\SOST_{\F}(\vd))\}.
  \]
Finally, we show that inclusion is strict when $m \ge 2$.
  For $m=1$, $\SOST_{\F}(d)$ is the cone of PSD matrices, which is self-dual.
  For $m \ge 2$, however, $\SOST_{\F}(\vd)$ is the cone of separable tensors, while its dual consists of block-positive tensors (operators that are nonnegative on all product vectors).
  It is a standard fact that block-positivity strictly relaxes separability: there exist block-positive operators that are not separable; see the example in \cite{peres1996separability}.
\end{proof}

Unlike the cone of PSD matrices, $\SOST_{\F}(\vd)$ is not self-dual.
In quantum information theory, the element of  $\SOST^*_{\F}(\vd)$ is called \textit{entanglement witness}.
An entanglement witness $\y^*$ is an operator that is nonnegative on all separable states and negative on at least one entangled state, thereby certifying entanglement when detected.
They can be analytically constructed for arbitrary pure states~\cite{Bourennane2004experimental} but often fall short when dealing with general (mixed) entangled states~\cite{Bertlmann2005optimal,Bertlmann2002geometric,Pittenger2001convexity}, which is
\begin{equation}
    \y^* = \sigma - \rho + \inner{\sigma}{\rho - \sigma}  \Id\,,
\end{equation}
where $\rho$ is the target entangled state, and $\sigma$ is its closest separable state in $\sep(\vd)$ (normalization omitted).
The main difficulty lies in finding $\sigma$ (which can be as hard as solving SDTO) and parametrizing $\rho$.

\begin{remark}
    \label{rmk.pair}
    The map $\Psi(\x)$ of a point $\x$ in the lift   $\F^{r\bar{d}}$ is defined as the sum of rank-1 extreme rays;  any point in  $\SOST_{\F}(\vd)$ is the sum of its extreme rays. This pairing implies that one can freely convert a point in the lift to a finite subset of extreme rays, and vice versa.
\end{remark}

\subsection{Lower and upper bounds for SDTO}
We present methods for constructing lower and upper bounds for the problem \eqref{eq:P}.

    \begin{lemma}
Any feasible point $\y^*$ in \eqref{eq:D} provides a lower bound $\inner{-\a}{\y^*}$ for $\opt$. Any feasible point $\y$ in \eqref{eq:P} provides an upper bound $\inner{\bc}{\y}$ for $\opt$.
\end{lemma}
\begin{proof}
    The feasible point in the dual problem gives a dual lower  bound, and the feasible point the primal problem gives a dual upper bound.
\end{proof}
We will use LADMM to find heuristic solutions for the nonconvex lifted problems,  where $\SOST_{\F}(\vd)$ is replaced by the smooth map $\Psi$ from a more tractable set $\F^{r\bar{d}}$. This heuristic method gives us a heuristic upper bound in the following senses.

Let $\x$ be an heuristic solution to the lifted problem of \eqref{eq:P}:
\begin{equation}
   \label{eq.LSDT} \tag{LSDTO}
    \min\{\inner{\bc}{\Psi(\x)} \mid \x \in \F^{r\bar{d}}   \;, \z \in \cZ \;,  A(\z) + \a  = \Psi(\x) \},
\end{equation} then $\ub_{\heur} = \inner{\bc}{\Psi(\x)}$ is a heuristic upper bound on $\opt$.

We present a relaxation method to compute numerically reliable lower and upper bounds.
We can rewrite \eqref{eq:D} in a semi-infinite programming formulation:

\begin{equation}
    \label{eq.semiD} \tag{SemiDSDTO}
   \opt= \max \{\inner{-\a}{\y^*} \mid \forall \y \in \ext(\SOST_{\F}(\vd)),\;  \inner{\y^\ast}{\y} \ge 0  \;,   - A^*(\y^*) + \bc\in \cZ^* \}
\end{equation}
\eqref{eq.semiD} has an infinite number of linear constraints.

A tractable approximation of \eqref{eq.semiD} is to replace the infinite number of linear constraints with a finite number of linear constraints, which can be done by generating a finite subset $\cP$ from $\ext(\SOST_{\F}(\vd))$. This leads to a finite-dimensional convex relaxation for \eqref{eq.semiD} and an upper bound on $\opt$:
\begin{equation}
    \label{eq:rD2} \tag{RelDSDTO}
 \opt \le  \ub_{\relx} \deq \max  \{\inner{-\a}{\y^*} \mid \forall \p \in \cP,\;  \inner{\y^\ast}{\p} \ge 0  \;,   - A^*(\y^*) + \bc\in \cZ^* \}
\end{equation}
Thus, a viable strategy is to solve the finite-dimensional convex relaxation \eqref{eq:rD2} to obtain an upper bound and refine the relaxation by the CP algorithm \cite{kortanek1993central}. The separation subproblem of \eqref{eq:rD2} asks for a linear constraint $\inner{\y^\ast}{\p} \ge 0$ violated by a given $\y^\ast$, so we can find such a constraint by solving the problem $\min_{\y \in \ext(\SOST_{\F}(\vd))}\inner{\y^\ast}{\y}$. Since $\SOST_{\F}(\vd)$ is convex, this problem is equivalent to a linear minimization problem:
\begin{equation}
    \label{eq.LO} \tag{LM}
    \min_{\y \in \SOST_{\F}(\vd)}\inner{\y^\ast}{\y}.
\end{equation}
 This problem is simpler than \eqref{eq:P}, but still requires LMO.

\begin{proposition}
Let $\y^\ast$ be a solution of the relaxed dual \eqref{eq:rD2} and let $\underbar{b}$ be a lower bound for \eqref{eq.LO}. Then, $\lb_{\relx}  \deq \ub_{\relx}  + \underbar{b}$ is a  lower bound on $\opt$.
\end{proposition}
\begin{proof}
    By Lagrangian duality of \eqref{eq:rD2}, we have that
    \begin{equation*}
      \opt = \max  \{\inner{-\a}{\y^*} + \min_{\y \in \ext(\SOST_{\F}(\vd)) } \inner{\y^\ast}{\y} \mid   - A^*(\y^*) + \bc\in \mathcal \cZ^* \}
    \end{equation*}
Since $\underbar{b} \le  \min_{\y \in \SOST_{\F}(\vd) } \inner{\y^\ast}{\y}  = \min_{\y \in \ext(\SOST_{\F}(\vd)) } \inner{\y^\ast}{\y}$, we have that $\opt \ge \inner{-\a}{\y^*} +  \underbar{b} = \ub_{\relx}  + \underbar{b}$.
\end{proof}

The following \Cref{tab.bounds} summarizes the different types of bounds for SDT optimization problems discussed above. For each bound, we indicate the computation required and the conditions under which it matches the true optimal value $\opt$.

\begin{table}[h!]
\centering
\begin{tabular}{l|ccc}
\toprule
Bound &   Computation & When equal $\opt$? & Remark\\
\midrule
$\ub_{\heur}$ & nonconvex \eqref{eq.LSDT} & \eqref{eq.LSDT}  solved & Often heuristic \\
$\ub_{\relx}$  & convex \eqref{eq:rD2} & \eqref{eq:rD2} solved & - \\
$\lb_{\relx}$  & \eqref{eq:rD2}/\eqref{eq.LO} lower bound & Unknown & Lagrangian bound  \\
\bottomrule
\end{tabular}
\caption{Summary of bounds for SDTO.}
\label{tab.bounds}
\end{table}

 These bounds can be computed numerically under our assumptions of smooth parameterization and the existence of an LMO,  but the underlying optimization problems remain challenging. In particular, both $\lb_{\relx}$ and $\ub_{\relx}$ rely on solving the linear minimization subproblem, which is  computationally difficult (recall the \NPh{}ness of the BSS problem). We note that only $\ub_{\relx}$ and $\lb_{\relx}$ are available during the optimization iteration. Computing  $\ub_{\heur}$  requires finding a feasible solution to the nonconvex constraint in \eqref{eq.LSDT}, however, often a heuristic solution is available (thus the constraint can be violated).

\begin{example}
    \label{exp.ref3}
        Following \Cref{exp.ref1}, we derive explicit formulations for the problems that compute bounds in \Cref{tab.bounds}. \eqref{eq.LSDT} becomes:
    \begin{equation}
      \ub_{\heur} \approx  \min \left\{ \z \;\middle|\z \ge 0,\, \x \in \C^{r\bar{d}},\; (\Id / \bar{d} - \phi)\z + \phi = \Psi(\x)\right\}.
    \end{equation}
\eqref{eq:rD2} becomes:
    \begin{equation}
     \ub_{\relx} =   \max \left\{ \inner{\y^*}{-\phi} \;\middle|\y^* \in \herm^{\bar{d}},\; 1 +\inner{\y^*}{\phi - \Id}   \ge 0\;, \forall \p \in \cP,\; \inner{\y^*}{\p} \ge 0  \right\}.
    \end{equation}
\end{example}

\section{Algorithms for SDTO}
\label{sec.algo}

In this section, we introduce our solution framework: (i) LADMM (\Cref{sec.ladmm}) computes high-quality heuristic primal solutions via a smooth factorization of PSD tensors; (ii) CP algorithm (\Cref{sec.sbb}) refines the solutions and provides numerically reliable upper and lower bounds. These two modules are coordinated by an iterative-refinement loop (\Cref{sec.implem}).

\subsection{LADMM for nonconvex lifted problems}
\label{sec.ladmm}

Smooth reparameterization is already used for a simpler constrained optimization problem:
\begin{equation}
    \label{eq:UP} \tag{P}
 \min_{\y \in \cY} f(\y),
\end{equation}
where \(f\) is a smooth function and \(\cY\) is a ``difficult'' convex set in a vector space \(\cE\) (one can think of \(\cY\) as \(\SOST_{\F}(\vd)\)). If we consider  \(\cY\) only as a convex set in \(\cE\), then we do not have efficient algorithms to solve \eqref{eq:UP}.

The idea is to reformulate \eqref{eq:UP} as a smooth optimization problem over a ``simple'' set, i.e., a Riemannian manifold. Recall that a smooth manifold \(\cM\) is Riemannian if it is equipped with an inner product \(\langle \cdot, \cdot \rangle_{\x}\) on each tangent space \(\tang_{\x}\cM\), which varies smoothly with the point \(\x\in\cM\). The definition of a lift is as follows.
 \begin{definition}[\cite{levin2025effect}]
     \label{def.lift}
     A (smooth) lift (parameterization) of a set \(\cY\subset\cE\) is a Riemannian manifold \(\cM\) in an ambient vector space \(\Xi\) together with a smooth map \(\Psi:\Xi\to\cE\) such that \(\Psi(\cM)=\cY\).
 \end{definition}
We call \(\Psi\) the lifting map. The dimension of the lifted space $\Xi$  is usually larger than the dimension of \(\cE\) (overparameterization). Given a lift \(\cM\) of the set \(\cY\), we can solve the following lifted problem:
\begin{equation}
    \label{eq:R} \tag{LiftedP}
    \min_{\x \in \cM} g(\x)\deq f(\Psi(\x)).
\end{equation}
The manifold constraint \(\x\in\cM\) is often easy to handle in practice. For optimization on Riemannian manifolds, Riemannian optimization algorithms maintain feasibility throughout the iterations, and stationary points of \eqref{eq:R} can be found efficiently using off-the-shelf tools; see \cite{bergmann2022manopt}. Note that lifting can introduce spurious local minima, i.e., the image of a local minimum of \eqref{eq:R} under the lifting map may not be a local minimum of \eqref{eq:UP}.

\begin{remark}
    \label{rmk.sdp}
    The factorization lifting map \(\Psi\) in \eqref{eq.lift} yields the Burer–Monteiro factorization \(\Psi(\x)=\sum_{i\in[r]}\x_i\x_i^\dagger\) for SDP matrices. It has been shown that, under appropriate conditions, the corresponding lifted problem \eqref{eq:R} has no spurious local minima \cite{burer2005local,journee2010low}. In practice, one can adopt a low-rank approximation for SDP matrices and increase the factorization size \(r\) until it becomes compact; see \cite{burer2003nonlinear,boumal2016non,journee2010low,tang2024feasible2,tang2024feasible}. As discussed in \cite{levin2025effect}, when \(\cY\) is a set of tensors, lifting can introduce spurious local minima for the lifted problem.
\end{remark}

To address problems of the form \eqref{eq:UP} in the presence of additional complicating constraints, we consider a more general problem:
\begin{equation}
    \label{eq:C} \tag{C}
     \min\{f(\y) \mid \y \in \cY,\; \z \in \cZ,\; A(\z) + \a = \y\}.
\end{equation}
When \(f(\y)=\langle\bc,\y\rangle\), \(\cY=\SOST_{\F}(\vd)\), and \(\cZ\) is a tractable convex cone, then \eqref{eq:C} reduces to \eqref{eq:P}.

 Lifting and ALM methods were used for tackling the SDP with complicating constraints $A(\z) + \a = \y$  similar to \eqref{eq:C}; see
\cite{wang2023decomposition,wang2025solving,tang2024feasible}. Inspired by these results,
 we  design a lifted alternating direction method of
multipliers (LADMM) for solving \eqref{eq:C} that further exploits separable structures $\cY$ and $\cZ$.

Define the Lagrangian of \eqref{eq:C}:
 \begin{equation}
    \label{eq.L} \tag{L}
    L(\z,\y,\chi) = f(\z) + \inner{A(\z) +  \a - \y}{\chi},
 \end{equation}
and  its augmented Lagrangian:
 \begin{equation}
    \label{eq.AL} \tag{AL}
    L_\zeta(\z,\y,\chi) = L(\z,\y,\chi) + \frac{\zeta}{2}\norm{A(\z) +  \a - \y}_2^2,
 \end{equation}
where $\zeta > 0$ is a penalty parameter.

The conventional ALM method will minimize the augmented Lagrangian $ L_{\zeta}(\z,\y,\chi) $ w.r.t. $\z,\y$ for a fixed dual $\chi$ and update $\chi$ accordingly.  In ADMM, the subproblem is decomposed into two steps, in which the augmented Lagrangian is minimized w.r.t. $\y$ and $\z$ separately. The subproblem w.r.t. $\z$ is a tractable convex optimization, so we can deploy off-the-shelf convex optimization solvers. The subproblem w.r.t. $\y$ is the same as the format \eqref{eq:P}, although it is not tractable, we can reformulate it  to the following lifted problem:
\begin{equation}
    \label{eq.single}
     \min_{\x \in \cM}L_{\zeta}(\z,\Psi(\x), \chi).
\end{equation}

\begin{algorithm}
\caption{Lifted Alternating Direction Method of Multipliers (LADMM)}
\label{algo.alm}
\KwIn{ $\z_1 \in \cZ$, $\x_1 \in \cM$, $\chi_1 \in \cE'$, and parameters $\zeta_1 > 0$, $\zeta_{\min} > 0$, $\zeta_{\max} > 0$}
\KwOut{$(\z_{k+1},\x_{k+1},\chi_{k+1})$}
$k \leftarrow 1$\;
\While{stopping criteria are not met}{
    Solve the  subproblem w.r.t. $\x$  approximately:
    \begin{equation}
        \label{eq.alm_subx}
     \x_{k+1} \leftarrow \argmin_{ \x \in \cM} L_{\zeta_k}(\z_k,\Psi(\x),\chi_k)
    \end{equation}\;
    Solve the  subproblem w.r.t. $\z$  exactly:
    \begin{equation}
        \label{eq.alm_subz}
     \z_{k+1} \leftarrow \argmin_{\z \in \cZ} L_{\zeta_k}(\z,\Psi(\x_{k+1}),\chi_k)
    \end{equation}\;
    $\chi_{k+1} \leftarrow \chi_k + \zeta_k(A(\z_{k+1}) + \a - \Psi(\x_{k+1}) )$\;
     update the penalty parameter $\zeta_{k+1}$\;
    $k \leftarrow k + 1$\;
}

\Return{$(\z_{k+1},\x_{k+1},\chi_{k+1})$}\;

\end{algorithm}

The LADMM in \Cref{algo.alm} is closely related to the approach combining ALM with Burer-Monteiro factorization for constrained low-rank SDPs; see \cite{wang2025solving}. The subproblem w.r.t. $\x$ is a standard unconstrained nonconvex optimization in the lifted space $\F^{r\bar{d}}$ given in  \eqref{eq.lift}, so advanced manifold concepts (e.g., retraction, Riemannian gradient) are unnecessary. If the subproblems are solved to global optimality and standard regularity conditions hold, one can invoke ADMM/ALM convergence results for convex components, see \cite{boyd2011distributed,deng2016global,he20121}.

In general, the subproblem \eqref{eq.alm_subx} cannot be solved to optimality. Thus, we can use LADMM  as a heuristic to find heuristic solutions for the problem \eqref{eq:P}.

 \begin{example}
    \label{exp.ref2}
Consider \eqref{eq:P}  in  \Cref{exp.ref3}. Note that the subproblem of minimization of $L_{\zeta}(\z,\y,\chi)$ w.r.t. $\z$ is a univariate quadratic problem, so the closed-form solution exists. Available off-the-shelf solvers can be used to solve the subproblem w.r.t. $\x$ inexactly, \eg Riemannian trust region (\cite{absil2007trust}) and Riemannian quasi-Newton (\cite{huang2015broyden}). In the case of \Cref{exp.ref3}, the manifold $\cM$ is simply the complex vector space $\C^{r\bar{d}}$ equipped with the standard Hermitian inner product.
 \end{example}

There are additional practical issues of  LADMM that inherit from ADMM. First, the constraint $A(\z) +  \a =  \y$ in \eqref{eq:C} can be violated by the solution. Second, it is known that ADMM can be very slow to converge to
high-accuracy solutions \cite{boyd2011distributed}.

\subsection{Iterative refinement algorithm}

We present the IR algorithm that improves solutions produced by the LADMM.

The IR procedure exploits a natural pairing between relaxations and factorizations. The bounds in \Cref{tab.bounds} can be interpreted via this pairing (see also \Cref{rmk.pair}). From any $\x\in\F^{r\bar{d}}$ one can reconstruct a finite set of extreme rays
\[
R(\x) \deq \{\p_j\}_{j\in[r]},\qquad
\p_j = \x_{j1}\mtens \x_{j1}^\dagger\mtens\cdots\mtens\x_{jm}\mtens\x_{jm}^\dagger,
\]
so that $\Psi(\x)=\sum_{\p\in R(\x)}\p$. Thus $R(\x)$ is a finite subset of $\ext(\SOST_{\F}(\vd))$.

To relate the upper bounds produced by the heuristic and by the relaxation, consider \eqref{eq.LSDT} and \eqref{eq:rD2}. The dual of \eqref{eq:rD2} is
\begin{equation}
    \label{eq:rD2D} \tag{DRelDSDTO}
   \ub_{\relx}= \min\{\inner{\bc}{\y}: \z\in\cZ,\; A(\z)+\a=\sum_{\p\in\cP}\lambda_\p\p,\;\lambda_\p\ge0\}.
\end{equation}

\begin{lemma}
    \label{lem.pair}
    Let $\x\in\F^{r\bar{d}}$ be feasible for \eqref{eq.LSDT}, and set $\cP=R(\x)$. Denote by $\ub_{\heur}=\inner{\bc}{\Psi(\x)}$ the heuristic upper bound and by $\ub_{\relx}$ the upper bound from \eqref{eq:rD2}. Then $\ub_{\heur}\ge\ub_{\relx}$. Moreover, if the elements of $R(\x)$ are affinely independent and $A$ is injective, then $\ub_{\heur}=\ub_{\relx}$.
\end{lemma}
\begin{proof}
By construction $\Psi(\x)=\sum_{j=1}^r\p_j$ provides a feasible representation for \eqref{eq:rD2D} (with $\lambda_j=1$), so $\ub_{\heur}=\inner{\bc}{\Psi(\x)}\ge\ub_{\relx}$. If the $\p_j$ are affinely independent, then the coefficients $\lambda_j$ are unique and hence equal to one; injectivity of $A$ makes $\z$ unique, yielding equality.
\end{proof}

Thus the factorization maintained by LADMM supplies a natural finite candidate set of extreme points for the dual relaxation \eqref{eq:rD2}.

We implement this idea using a CP procedure that maintains a finite inner approximation $\cP\subseteq\ext(\SOST_{\F}(\vd))$. Because $\SOST_{\F}(\vd)$ is intractable, we assume access to an LMO for \eqref{eq.LO} that, on input $\y$, returns a lower bound $\underbar{b}$ and an upper bound $\bar b$ for the separation problem; the oracle also returns an extreme point $\p$ stored as the form $ v_{1}\mtens v_{1}^\dagger\mtens\cdots\mtens v_{m}\mtens v_{m}^\dagger$ with the same value as the upper bound $\underbar{b}$.

Each CP iteration solves the LP \eqref{eq:rD2} over the current set $\cP$ to obtain $\ub_{\relx}$ and a corresponding dual solution $(\{\lambda_\p\}_{\p\in\cP},\z)$ in $\eqref{eq:rD2D}$. The LMO is then called to search for a violated cut. If the LMO returns $\underbar b\ge0$, no violated cut exists and the current $\y$ is optimal for the original SDT problem, so $\ub_{\relx}= \lb_{\relx} = \opt$. If $\bar b<0$, the extreme point $\p$ is added to $\cP$ and the algorithm continues. If the oracle cannot produce a valid extreme point (generation fails), the procedure terminates without any further improvement of the bounds. We also maintain a numerically reliable (Lagrangian) lower bound $\lb_{\relx}=\ub_{\relx}+\underbar b$ throughout the iterations.

The standard CP algorithm is slow \cite{kortanek1993central}. Thus, the IR algorithm alternates between (i)  solving \eqref{eq.LSDT} using LADMM and (ii) solving or refining the dual relaxation \eqref{eq:rD2} using the CP procedure. In practice, this alternating refinement often improves both feasibility and objective quality compared with LADMM alone.

\begin{algorithm}
\caption{Cutting-plane (CP) algorithm}
\label{algo.dual}
\KwIn{Initial set $\cP_1 \subseteq \ext(\SOST_{\F}(\vd))$}
\KwOut{Feasible solution (or certificate) for \eqref{eq:P}}
$k \leftarrow 1$\;
\Repeat{stopping criteria are met}{
    solve \eqref{eq:rD2} over $\cP_{k}$ and obtain $\ub_{\relx}$\;
    recover a corresponding dual solution $(\{\lambda_\p\}_{\p \in \cP_k},\z_k)$ for \eqref{eq:rD2D}\;
    call the LMO for \eqref{eq.LO}; obtain a lower bound $\underbar{b}$, an upper bound $\bar{b}$, and an extreme point $\p'$\;
    update $\lb_{\relx} \leftarrow \ub_{\relx}  + \underbar{b}$\;
    \eIf{$\underbar{b} \ge 0$}{
        The current solution $\sum_{\p\in\cP_k}\lambda_\p\p$ is optimal for \eqref{eq:P}; \textbf{break}\;
    }{
        \eIf{$\bar{b} < 0$}{
            add $\p'$ to $\cP_k$\;
        }{
            generation failed (no violated cut found); \textbf{break}\;
        }
    }
    $k \leftarrow k + 1$\;
}
\Return{$(\{\lambda_\p\}_{\p \in \cP_k}, \cP_k, \z_k)$}\;
\end{algorithm}

The overall IR procedure is given in Algorithm \ref{algo.lift}. The algorithm repeatedly calls the CP routine after the LADMM terminates. The LADMM output is used to warm-start the CP, and conversely the CP output can be used to warm-start LADMM. This interplay helps avoid infeasible solutions and poor local minima returned by LADMM alone, and slow convergence of CP. In step 4, let us write $\cP_k$ as $ \{\p_j\}_{j \in [r]}$, as $\p_j$ is stored as $ v_{j1}\mtens v_{j1}^\dagger\mtens\cdots\mtens v_{jm}\mtens v_{jm}^\dagger$ and set $\x_k$  as $(v_{j1}\lambda^{1/{2m}}_j,\ldots,v_{jm}\lambda^{1/{2m}}_j)_{j \in [r]}$ such that $\Psi(\x_k) = \sum_{\p_j\in\cP_k}\lambda_j \p_j$.

\begin{algorithm}
    \caption{Iterative refinement (IR) algorithm}
    \label{algo.lift}
    \KwIn{Objective $f$, set $\SOST_{\F}(\vd)$ with lift $\cM$, cone $\cZ$}
    \KwOut{A feasible solution for \eqref{eq:P} (with bounds)}
    construct initial $\z_1\in\cZ$, initial active set $\cP_1\subseteq\ext(\SOST_{\F}(\vd))$, and weights $\{\lambda_{\p}\}_{\p\in\cP_1}$\;
    $k \leftarrow 1$\;
    \Repeat{stopping criteria are met}{
        form $\x_k \in \F^{rm\bar d}$ corresponding to $\sum_{\p\in\cP_k}\lambda_{\p}\p \in \SOST_{\F}(\vd)$\;
        call \Cref{algo.alm} (LADMM) with warm start $(\z_k,\x_k)$ to obtain $(\z'_k, \x'_k)$\;
        let $\cP'_k \leftarrow \cP_k \cup R(\x'_k)$ (union of previous active set and active components from $\x'_k$)\;
        call \Cref{algo.dual} (CP) with warm start $\cP'_k$ to obtain $(\lambda_{\cP''_k},\cP''_k,\z_k)$\;
        update global bounds $\lb_{\relx}$ and $\ub_{\relx}$\;
        set $\cP_{k+1} \leftarrow \{\p\in\cP''_k\mid \lambda_{\p}>0\}$ (keep active components)\;
        $k \leftarrow k+1$\;
    }
    \Return{$\sum_{\p\in\cP_k}\lambda_{\p}\p $ as a final solution (and recorded bounds)}\;
\end{algorithm}

The CP subroutine (\Cref{algo.dual}) ensures that the upper bound is non-increasing, and the IR loop keeps the active set size manageable.

\begin{proposition}
    Assume $\dim(\SOST_{\F}(\vd)) = D$. The IR loop guarantees that the upper bound $\ub_{\relx}$ is non-increasing across iterations, and the cardinality $|\cP'_k|$ of the intermediate active set is at most $ 4D+2$.
\end{proposition}
\begin{proof}
    In each iteration, the CP routine returns an optimal solution over the current active set; when this solution is reused to form the warm start for the next iteration, the convex hull of the new active set must contain at least one previously computed optimal solution. Hence, the upper bound cannot increase and will either decrease or remain the same. For the cardinality bound, note $|\cP'_k|\le |R(\x'_k)| + |\cP_k|$. By Proposition \ref{prop.compact}, any tensor in $\SOST_{\F}(\vd)$ admits a representation with at most $2D+1$ components, so both $|R(\x'_k)|$ and $|\cP_k|$ are bounded by $2D+1$, implying $|\cP'_k|\le 4D+2$.
\end{proof}

\begin{remark}
    \label{rmk.crossover}
A practical crossover strategy for the LADMM is to set the iteration limits of both \Cref{algo.lift} and \Cref{algo.dual} to one, then construct the convex combination of the active set returned by CP together with $R(\x)$ from LADMM. Finding a feasible point in that convex hull provides a simple and effective way to turn the LADMM output into a certified feasible solution.
\end{remark}

We defer the implementation details and parameter choices for these algorithms to \Cref{sec.implem}.

\section{Branch-and-bound LMO}

\label{sec.sbb}

For \NPh{} problems, a common approach is to employ oracle-based optimization algorithms that decompose the original problem into simpler—though typically still \NPh{}—subproblems. This idea underlies methods such as Dantzig–Wolfe decomposition for integer programming \cite{Vanderbeck2010} and CG methods \cite{cgbook}. Due to the \NPh{} nature of the subproblems and performance considerations, oracles are often exhaustive enumeration with time limits or heuristics. However, an oracle capable of exploring the entire search space is generally required when heuristics fail to find a solution. Similarly, the CP procedure \Cref{algo.dual} is based on an LMO.

In \cite{liu2025unified}, the LMO constructs a simplicial partition of the pure product states and searches for the best solution within this a priori partition. Below, we describe an sBB-based LMO, in which the partition is realized lazily through the branching rule. We assume that readers are familiar with sBB; see \cite{belotti2009branching} for the general approach and \cite{chen2017spatial} for a similar pipeline in complex nonconvex second-order polynomial optimization. We summarize only the modules essential in our implementation: convex relaxations, branching rules, and heuristics.

W.l.o.g., we can normalize the tensor in the separation subproblem \eqref{eq.LO} to unit trace. We focus on the complex field $\F=\C$ (the real case $\F=\R$ is analogous and simpler). With this normalization, the separation problem \eqref{eq.LO} becomes the BSS problem:
\begin{equation}
    \label{eq.sep} \tag{BSS}
        \min_{\y \in \sep(\vd)}\inner{\chi}{\y}.
\end{equation}
Omitting the constraint $\tr(\y)=1$ in \eqref{eq.sepstate}  defining $\sep(\vd)$ yields the unnormalized BSS (which is unbounded); the two formulations are equivalent up to scaling, and we focus on the normalized one.

\subsection{Convex relaxations}
\label{sec.rel}

The sBB algorithm requires a convex outer approximation $\cC$ of separable states for constructing convex relaxations of the BSS problem. In addition, the set $\cC$ yields a convex relaxation of  \eqref{eq:P}:
\begin{equation}
\label{eq:PR}
     \min_{\z \in \cZ,\; A(\z) + \a \in \cC} \langle\bc,\z\rangle,
\end{equation}
   whose optimal value is a valid lower bound for \eqref{eq:P}.

In definition \eqref{eq.sepstate}, separable states are expressed as convex combinations of rank-1 tensor products.
Our convex outer approximations are based on an equivalent, often more convenient, lifted representation of separable states.

\begin{proposition}
\label{prop.state}
    \[
        \sep(\vd) \;=\; \conv\bigl\{ \bmtens_{k\in[m]}\X_k \;:\; \X_k\in\dmat^{d_k}\ \text{for all }k\in[m]\bigr\}.
    \]
\end{proposition}
\begin{proof}
    The inclusion $\sep(\vd)\subseteq\conv\{\bmtens_{k}\X_k\}$ is immediate from the definition. For the reverse inclusion, let
    $\rho=\sum_{i=1}^n\alpha_i \rho_i$ be any convex combination with $\rho_i=\bmtens_{k}\X_{ik}$ and $\alpha_i\ge0$, $\sum_i\alpha_i=1$.
    By the spectral theorem, each density matrix $\X_{ik}$ admits a decomposition into rank-1 density matrices:
    \[
        \X_{ik} \;=\; \sum_{j=1}^{d_k} \beta_{ikj}\, v_{ikj}v_{ikj}^\dagger,
        \qquad \beta_{ikj}\ge0,\ \sum_j\beta_{ikj}=1,
    \]
    where $v_{ikj}$ are unit vectors. Taking tensor products of these decompositions yields that each $\rho_i$ is a convex combination of pure product states (rank-1 tensor products). Hence $\rho$ itself is a convex combination of pure product states, proving the claim.
\end{proof}

Thus,
in this section, we lift the tensor product $\x_k\x_k^\dagger$ to $\X_k$.
This proposition implies that any convex constraints satisfied by tensors of the form $\bmtens_{k}\X_k$ yield a convex outer approximation of $\sep(\vd)$. Working with the subsystems' density matrices $\X_k$ is often more tractable than manipulating rank-1 product states directly. After reviewing the DPS outer approximations of  separable bipartite states in \Cref{sec.rel1}, we extend them to construct SDP  outer approximations for separable multipartite states in \Cref{sec.rel2}. We also present linear inequalities useful for spatial branching in \Cref{sec.rel3}, and in \Cref{sec.rel4} we derive additional valid SDP inequalities via a tensor RLT that unifies these constructions.

\subsubsection{DPS outer approximations  of bipartite systems}
\label{sec.rel1}

The DPS hierarchy \cite{doherty2004complete} was proposed for separable bipartite states. Consider the dimension vector $d_{AB}=(d_A,d_B)$ of a bipartite system and the tensor product
\begin{align}\label{sep state}
    \rho_{AB} = \rho_A  \mtens \rho_B,
\end{align}
of two density matrices $\rho_A \in \dmat^{d_A}$ and $\rho_B \in \dmat^{d_B}$. The DPS hierarchy is based on an extension of $\rho_{AB}$ on the $B$ subsystem:
\begin{equation}
\label{eq:stateextension}
    \rho_{AB_{[\ell]}} =  \rho_A  \mtens  (\rho_{B})^{\mtens \ell},
\end{equation}
for $\ell \ge 1$.

The extended state $\rho_{AB_{[\ell]}}$ satisfies the following constraints:
\begin{itemize}
\item[(a)] Positivity: $\rho_{AB_{[\ell]}} \succeq 0$.
\item[(b)] Partial-trace projection: tracing out the systems $B_2,\ldots,B_\ell$ from $\rho_{AB_{[\ell]}}$ recovers the original bipartite state:
\begin{equation}
\label{eq:extensionpartialtracecond}
\tr_{B_{[2:\ell]}} (\rho_{AB_{[\ell]}}) =   \rho_A  \mtens \rho_{B} = \rho_{AB}.
\end{equation}
Here $\tr_{B_{[2:\ell]}}$ denotes the partial trace over subsystems $B_2,\dots,B_\ell$ and this operator is linear.
\item[(c)] Symmetry: the extension is symmetric under permutations of the $\ell$ copies of subsystem $B$. Equivalently, if $\Pi$ is the orthogonal projector onto the symmetric subspace of ${\herm^{d_B}}^{\otimes \ell}$, then
\begin{equation}
\label{eq:extensionsymmetrycond}
(\Id \otimes \Pi)\,\rho_{AB_{[\ell]}}\,(\Id\otimes \Pi) = \rho_{AB_{[\ell]}}.
\end{equation}
\item[(d)] PPT: for any subset of the $B$-copies, the partial transpose with respect to that subset yields a PSD matrix. In particular, for any $s=1,\ldots,\ell$ and any choice of $s$ copies,
\begin{equation}
\label{eq:extensionpptcond}
\Tr_{B_{[s]}}(\rho_{AB_{[\ell]}}) \succeq 0,
\end{equation}
where $\Tr_{B_{[s]}}$ denotes the partial transpose on the selected $B$-subsystems; see \eqref{eq.transpose} for the elementwise definition.
\end{itemize}

The PSD variable $\rho_{AB_{[\ell]}}$ lives in an extended space of dimension $(d_A d^\ell_B)^2$. Define
\begin{equation}
 \dps_\ell(d_{AB}) \;=\; \{\,w_{AB} \in \dmat^{d_{AB}}:\ \exists\ \rho_{AB_{[\ell]}}\ \text{satisfying } \eqref{eq:extensionpartialtracecond},\ \eqref{eq:extensionsymmetrycond},\ \eqref{eq:extensionpptcond}\,\}.
\end{equation}
We refer to the collection of conditions above as the level-$\ell$ DPS constraints, which are completely determined by the configuration $(d_A,d_B, \ell)$. The DPS hierarchy yields a nested family of convex outer approximations of $\sep(d_{AB})$:
\begin{equation}
    \sep(d_{AB}) \subseteq \cdots \subseteq \dps_\ell(d_{AB})  \subseteq  \cdots   \subseteq \dps_1(d_{AB}).
\end{equation}

The DPS constraints are often stated without explicitly enforcing unit-trace constraints on the extended variables;  unit-trace constraints on the extended variables are often added in practice.

\begin{remark}
    \label{rmk.dps}
 The DPS hierarchy can also be used to construct convex outer approximations for separable multipartite states by grouping subsystems into a bipartition. For example, writing a tripartite state as the bipartition $\rho_{ABC}=\rho_A\otimes\rho_{BC}$ allows one to relax it as a bipartite system for the pair $(A'=A,B'=BC)$.  Thus, the DPS constraints corresponding to the configuration $(d_{A'},d_{B'}, \ell)$  give a convex outer approximation of $\cC$ for separable multipartite states. This approach is available in Ket.jl.
\end{remark}

\subsubsection{Factorizable DPS outer approximations for multipartite systems}
\label{sec.rel2}
We build convex outer approximations of separable multipartite states via factorization.
 Because the tensor product is associative, there are several ways to evaluate the expression $\y = \bmtens_{k\in[m]}\X_k$. One natural choice is to evaluate from left to right. Introduce auxiliary variables
\[
\Z_h \deq \bmtens_{k \in [h]} \X_k,
\]
with $\Z_1=\X_1$, so that $\Z_m=\y$ and the recursion $\Z_k = \Z_{k-1} \mtens \X_k$ holds. This factorization yields the following recursive DPS outer approximation for the multipartite states in $\sep(\vd)$:
\begin{subequations}
\label{eq.rmmr}
\begin{align}
&     &     \Z_{1} = & \X_1, \\
& \forall k \in[m],        &  \X_k \succeq & 0, \\
& \forall k \in[m],      &  \tr(\X_k) = & 1,\\
& \forall k \in[2:m],     &     \tr(\Z_{k}) = & 1, \\
& \forall k \in[2:m],      &     \Z_{k} \in  & \dps_{\ell}((d_{[k-1]},d_k)), \\
& \forall k \in[2:m],      &  \tr_{\X_k}(\Z_{k}) = & \Z_{k-1}, \\
& \forall k \in[2:m],      &  \tr_{\Z_{k-1}}(\Z_{k}) = & \X_k,
\end{align}
\end{subequations}
where the linear maps $\tr_{\X_k}$ and $\tr_{\Z_{k-1}}$ trace out the subsystems corresponding to $\X_k$ and $\Z_{k-1}$ as in \eqref{eq:extensionpartialtracecond}. These projection constraints link each extended state $\Z_k$ with the projected states $\X_k$ and $\Z_{k-1}$, enabling recovery of $\X_k$ from a relaxation solution.

To further reduce the number of PSD variables, we present a dichotomy DPS (DDPS) outer approximation by exploiting the associativity of tensor product. Represent the tensor product evaluation by a binary tree $T=(S,E)$ whose nodes $S$ correspond to intermediate states and states $\X_k$; each node $k\in S$ is associated with a variable $\Z_k$. Leaf nodes $S'=[m]\subseteq S$ correspond to the subsystem variables $\X_k$, and the root node $k_0$ corresponds to $\y$. Each internal node has two subnodes $k_A,k_B\in S$ and a PSD variable $\Z_k$ that equals the tensor product of its subnodes' variables, i.e., $\Z_k=\Z_{k_A}\mtens\Z_{k_B}$. The DDPS relaxation is given by:
\begin{subequations}
\label{eq.dmmr}
\begin{align}
&        &   \Z_{k_0} = & \y,\\
& \forall k \in S',     &   \Z_{k} = & \X_k, \\
& \forall k \in S,        &  \Z_k \succeq & 0, \\
& \forall k \in S,      &  \tr(\Z_k) = & 1,\\
& \forall k \in S\setminus S',      &     \Z_{k} \in  & \dps_{\ell}((d_{k_A},d_{k_B})), \\
& \forall k \in S\setminus S',     &     \tr_{\Z_{k_A}}(\Z_{k}) = & \Z_{k_B}, \\
& \forall k \in S\setminus S',     &   \tr_{\Z_{k_B}}(\Z_{k}) = & \Z_{k_{A}},
\end{align}
\end{subequations}
where $d^2_{k_A}$ and $d^2_{k_B}$ denote the dimensions of subsystems $\Z_{k_A}$ and $\Z_{k_B}$, respectively.

The recursive system \eqref{eq.rmmr} is a special case of the dichotomy system \eqref{eq.dmmr}, corresponding to a path-shaped tree in which each internal node has at least one subnode that is a leaf. If the tree is balanced, its depth can be reduced to $\mathcal{O}(\log m)$ and the number of additional PSD variables is minimized.

\begin{proposition}
    Both \eqref{eq.rmmr} and \eqref{eq.dmmr} define convex outer approximations of the separable multipartite set.
\end{proposition}
\begin{proof}
    It suffices to prove for \eqref{eq.dmmr}. It is easy to show that \eqref{eq.dmmr} relaxes the lifted reformulation
    \begin{subequations}
\begin{align}
&        &   \Z_{k_0} = & \y,\\
& \forall k \in S',     &   \Z_{k} = & \X_k, \\
& \forall k \in S,        &  \Z_k \succeq & 0, \\
& \forall k \in S,      &  \tr(\Z_k) = & 1,\\
& \forall k \in S\setminus S',      &     \Z_{k}=  &  \Z_{k_A} \mtens \Z_{k_B}.
\end{align}
of $\y = \bmtens_{k \in [m]} \X_k$.
\end{subequations}
\end{proof}

The level of the DPS outer approximation is determined by the level $\ell$ used in each $\dps_{\ell}$ constraint.   The experiments in \Cref{sec.numerical} show that, even for relatively small bipartite systems such as $d_{AB}=(8,4)$, taking $\ell=3$ already produces SDPs that exceed the memory limits of some solvers. Hence, high-level DPS constraints are not scalable for use inside an sBB algorithm.

In our implementation, we use the first-level DDPS outer approximation (implemented via a balanced binary tree) for the SDP relaxation inside the sBB LMO.  In this way, the lower bounds and convex relaxations of \eqref{eq.sep} are tightened by partitioning the variable space.

\subsubsection{Complex McCormick inequalities for multipartite systems}
\label{subsec.complex}
\label{sec.rel3}
The conventional sBB method typically branches on univariate or bilinear expressions so that the relaxations of these expressions tighten as variable ranges shrink \cite{belotti2009branching}. The DPS constraints derived from unnormalized separable states do not depend on variable bounds. Therefore, we choose to work with the normalized BSS formulation, which allows us to derive variable bounds and LP relaxations that depend on those bounds.

 The bilinear equation $\Z_k = \Z_{k_A} \mtens \Z_{k_B}$ generally does not hold in the DDPS system \eqref{eq.dmmr}, and  we propose valid LP inequalities that relax the entry-wise bilinear equation in $\Z_k = \Z_{k_A} \mtens \Z_{k_B}$. These linear inequalities can be tightened as variable bounds are reduced by the branching rule that enforces the bilinear equation.

For any Hermitian matrix $\Z$, we separate real and imaginary parts and write $\Z = \U + \vi \V$, where $\U$ is real symmetric and $\V$ is real skew-symmetric. We first state a useful implication about density matrices.

\begin{theorem}
\label{thm.implication}
Let $\Z = \U + \vi \V \in \herm^d$. Consider the following statements:
\begin{enumerate}
    \item $\Z \succeq 0$ and $\tr(\Z) = 1$;
    \item $\Z \succeq 0$ and $\Z \preceq \Id$;
    \item every $2\times2$ principal submatrix $\Z' = \U' + \vi \V'$ of $\Z$ satisfies $\Z' \succeq 0$ and $\Z' \preceq \Id$;
    \item for every $2\times2$ principal submatrix $\Z' = \U' + \vi \V'$ of $\Z$ we have $0 \le U'_{ii} \le 1$, $V'_{ii} = 0$ for $i\in\{1,2\}$, and $-1 \le V'_{ij} \le 1$, $-1 \le U'_{ij} \le 1$ for $i\ne j$.
\end{enumerate}
Then the implications $1 \implies 2 \implies 3 \implies 4$ hold.
\end{theorem}

\begin{proof}
If $\Z\succeq0$ and $\tr(\Z)=1$, diagonalize $\Z$ by a unitary $T$ so that $T\Z T^\dagger$ is diagonal with entries in $[0,1]$. Hence $T\Z T^\dagger \preceq \Id$, and by congruence with a unitary we get $\Z \preceq \Id$. This proves $1\!\implies\!2$.

If $\Z \succeq 0$ and $\Id-\Z\succeq0$, then every principal $2\times2$ submatrix of both matrices is PSD, so $2\!\implies\!3$.

Assume $3$ holds.  Note that  $U'_{ii} \ge 0, V'_{ii} = 0$ ($i \in \{1,2\}$) follows from the Hermitian PSD of $\U'+\vi \V'$. Let $E_i$ be the diagonal matrix with only $i$-th diagonal being one, then $ U'_{ii} = \tr[\Z' E_i ] \le \tr[\Id E_i] = 1$. Let the all-ones matrix be $E$, then $0  = \tr[0  E] \le  \tr[\Z' E]   \le \tr[\Id  E]  =  2 $. Since $\tr[\V'  E] = 0$, $ \tr[\Z' E]   = \tr[\U' E]   =  U'_{11} + U'_{22} + 2 U'_{12}  $. Then, we have $-(U'_{11} + U'_{22}) /2 \le   U'_{12} \le (2 - (U'_{11} + U'_{22}))/2$. Let  $E'=[1, 1; -1, 1]$, then $0  = \tr[0  E'] \le  \tr[\Z' E']   \le \tr[\Id  E']  =  2 $. Then, $ \tr[\Z' E]   = \tr[\U' E]   =  U'_{11} + U'_{22} + 2 V'_{12}  $. Then, we have $-(U'_{11} + U'_{22}) /2 \le   V'_{12} \le (2 - (U'_{11} + U'_{22}))/2$. Then, the assumption that $\Z' \succeq 0$ and $\Z' \preceq \Id$ implies that
    \begin{equation}
    \label{eq.implieq}
            \begin{split}
            Z' \in \herm^d & \implies 0 \le V'_{ii} \le 0, \\
                \tr[0  E_i] \le \tr[\Z' E_i]   \le   \tr[\Id  E_i]  & \implies   0 \le   U'_{ii} \le 1, \\
    \tr[0  E] \le \tr[\Z' E]   \le   \tr[\Id  E]  & \implies  -1 \le  \frac{-U'_{11} - U'_{22}}{2} \le   U'_{12} \le \frac{2 - (U'_{11} + U'_{22})}{2} \le 1, \\
  \tr[0  E'] \le \tr[\Z' E']   \le \tr[\Id  E']   & \implies -1 \le \frac{-U'_{11} - U'_{22}}{2} \le   V'_{12} \le \frac{2 - (U'_{11} + U'_{22})}{2}  \le 1.
    \end{split}
    \end{equation}
     Thus, $3 \implies 4$.
\end{proof}

We now consider entries of $\Z_k=\U_k+\vi\V_k$, which satisfy condition 1 in \Cref{thm.implication}. Let $\overline{\U}_k,\overline{\V}_k$ (resp. $\underline{\U}_k,\underline{\V}_k$) denote entry-wise upper (resp. lower) bounds implied by condition 4. Define the index mapping $\delta(i_A,i_B)=i_A d_{k_B}+i_B$. The entries of $\Z_k$ can be written as bilinear products of entries of $\Z_{k_A}$ and $\Z_{k_B}$:
\begin{equation}
\label{eq.zk}
\begin{split}
& Z_{k,\delta(i_A,i_B),\delta(j_A,j_B)} \\
&\qquad = (U_{k_A,i_A j_A} + \vi V_{k_A,i_A j_A})(U_{k_B,i_B j_B} + \vi V_{k_B,i_B j_B}) \\
&\qquad = U_{k_A,i_A j_A}\,U_{k_B,i_B j_B} - V_{k_A,i_A j_A}\,V_{k_B,i_B j_B} \\
&\qquad\quad + \vi\bigl(V_{k_A,i_A j_A}\,U_{k_B,i_B j_B} + U_{k_A,i_A j_A}\,V_{k_B,i_B j_B}\bigr).
\end{split}
\end{equation}

The convex envelope of a real bilinear term $W_A W_B$ over the box $[\underline{W}_A,\overline{W}_A]\times[\underline{W}_B,\overline{W}_B]$ is given by McCormick inequalities \cite{mccormick1976computability}:
\begin{subequations}
\label{eq.mcccs}
\begin{align}
W_A W_B  &\ge  \overline{W}_B W_A + \overline{W}_A W_B - \overline{W}_A \overline{W}_B, \\
W_A W_B  &\ge  \underline{W}_B W_A + \underline{W}_A W_B - \underline{W}_A \underline{W}_B, \\
W_A W_B  &\le  \underline{W}_B W_A + \overline{W}_A W_B - \overline{W}_A \underline{W}_B, \\
W_A W_B  &\le  \overline{W}_B W_A + \underline{W}_A W_B - \underline{W}_A \overline{W}_B.
\end{align}
\end{subequations}
These inequalities can also be obtained via the scalar RLT procedure in \cite{sherali1992global}. Denote the sets of affine under- and over-estimators of $W_A W_B$ by $\underline{W_A W_B}$ and $\overline{W_A W_B}$, respectively.

For brevity set
\[
u_A = U_{k_A,i_A j_A},\quad u_B = U_{k_B,i_B j_B},\quad
v_A = V_{k_A,i_A j_A},\quad v_B = V_{k_B,i_B j_B},
\]
and
\[
\mu_{AB} = U_{k,\delta(i_A,i_B),\delta(j_A,j_B)},\quad
\nu_{AB} = V_{k,\delta(i_A,i_B),\delta(j_A,j_B)}.
\]
Applying McCormick to the four real bilinear terms $u_A u_B$, $v_A v_B$, $v_A u_B$, and $u_A v_B$ and combining them yields a collection of linear inequalities (aggregated McCormick inequalities):
\begin{subequations}
\label{eq.mccons}
\begin{align}
&\forall a\in\underline{u_A u_B},\ \forall a'\in\overline{v_A v_B}:\quad
\mu_{AB} \ge a(u_A,u_B) - a'(v_A,v_B), \label{eq.mca}\\
&\forall a\in\underline{u_A u_B},\ \forall a'\in\overline{v_A v_B}:\quad
\mu_{AB} \le a(u_A,u_B) - a'(v_A,v_B),\\
&\forall a\in\underline{u_A v_B},\ \forall a'\in\underline{v_A u_B}:\quad
\nu_{AB} \ge a(u_A,v_B) + a'(v_A,u_B),\\
&\forall a\in\overline{u_A v_B},\ \forall a'\in\overline{v_A u_B}:\quad
\nu_{AB} \le a(u_A,v_B) + a'(v_A,u_B).
\end{align}
\end{subequations}

Relaxing the entry-wise equations in $\Z_k=\Z_{k_A}\mtens\Z_{k_B}$ by these inequalities yields an LP relaxation for the BSS problem.

\subsubsection{Tensor RLT procedure and inequalities}
\label{sec.rel4}
We now introduce valid SDP constraints derived from a tensor RLT procedure for PSD variables. The construction generalizes the scalar RLT and is based on tensor products.

The tensor RLT on two SDP constraints appears in \cite{anstreicher2017kronecker,sturm2003cones}. We first review and apply it to the proposed relaxations, and then extend it for multiple SDP constraints. Let $F(\Z_{k_A})$ and $H(\Z_{k_B})$ be two linear matrix-valued maps acting on PSD variables $\Z_{k_A}$ and $\Z_{k_B}$. If
$F(\Z_{k_A}) \succeq 0$ and $H(\Z_{k_B}) \succeq 0$, then by Proposition 1.1 of \cite{anstreicher2017kronecker} the tensor product
$F(\Z_{k_A}) \mtens H(\Z_{k_B}) \succeq 0$ also holds. The linearization step rewrites
$F(\Z_{k_A}) \mtens H(\Z_{k_B})$ as a linear map in the variables $\Z_{k_A}$, $\Z_{k_B}$ and the auxiliary variable $\Z_k$ representing $\Z_{k_A}\mtens\Z_{k_B}$, yielding a valid SDP inequality. The tensor product $\Z_{k_A}\mtens\Z_{k_B}$ is then replaced by the auxiliary PSD variable $\Z_k$ (the so-called RLT variable).

Applying this construction to the PSD variables appearing in the dichotomy DPS relaxation leads to the following tensor McCormick inequalities:
\begin{subequations}
\label{eq.psd}
\begin{align}
    \Z_{k} &\preceq  \Id \mtens \Z_{k_B}, \\
     \Z_{k} &\preceq \Z_{k_A} \mtens \Id, \\
     \Z_{k} &\succeq \Id \mtens \Z_{k_B} + \Z_{k_A} \mtens \Id - \Id,
\end{align}
\end{subequations}
which are valid whenever $\Z_k=\Z_{k_A}\mtens\Z_{k_B}$ and the factors are density matrices.

\begin{proposition}
Given density matrices $\Z_k, \Z_{k_A}, \Z_{k_B}$, the tensor McCormick inequalities \eqref{eq.psd} are valid consequences of the equality $\Z_k=\Z_{k_A}\mtens\Z_{k_B}$.
\end{proposition}
\begin{proof}
By \Cref{thm.implication} we may use the factor constraints
$\Z_{k_A}\succeq0,\ \Z_{k_B}\succeq0,\ \Id-\Z_{k_A}\succeq0,\ \Id-\Z_{k_B}\succeq0$.
Taking tensor products of suitable factor combinations yields PSD constraints such as
\[
\Z_{k_A}\mtens\Z_{k_B}\succeq0,\qquad
(\Id-\Z_{k_A})\mtens\Z_{k_B}\succeq0,\qquad
\Z_{k_A}\mtens(\Id-\Z_{k_B})\succeq0,
\]
and so on. Expanding these tensor-factor constraints and then linearizing the terms that involve $\Z_{k_A}\mtens\Z_{k_B}$ produces matrix inequalities of the form \eqref{eq.psd}. Thus \eqref{eq.psd} follows from the factor constraints and the identity $\Z_k=\Z_{k_A}\mtens\Z_{k_B}$.
\end{proof}

 One can show that the tensor McCormick inequalities are stronger than the aggregated (entrywise) McCormick inequalities \eqref{eq.mccons}.

\begin{theorem}
If density matrices $\Z_k,\Z_{k_A},\Z_{k_B}$ satisfy the tensor McCormick inequalities \eqref{eq.psd}, then the aggregated McCormick inequalities \eqref{eq.mccons} also hold.
\end{theorem}
\begin{proof}
By symmetry it suffices to derive \eqref{eq.mca} from \eqref{eq.psd}. The aggregated McCormick inequalities are linearizations of bilinear factor constraints such as
$(u_A-\underline{u}_A)(u_B-\underline{u}_B)\ge0$ and similar expressions for the imaginary parts. According to the proof in \Cref{thm.implication}, each scalar bound like $u_A\ge\underline{u}_A$ is implied by trace inequalities applying suitable projections against $0 \preceq \Z_{k_A} \preceq \Id$ and $0 \preceq \Z_{k_B} \preceq \Id$; likewise for $u_B$. Suppose   $u_A\ge\underline{u}_A$ is implied by $\tr[E_{1}\Z_{k_A}] + \tr[E_2(\Id - \Z_{k_A})] + \tr[E_{2}\Z_{k_B}] + \tr[E_4(\Id - \Z_{k_B})] \ge 0$ and $u_B\ge\underline{u}_B$ is implied by  $\tr[E_{5}\Z_{k_A}] + \tr[E_6(\Id - \Z_{k_A})] + \tr[E_{7}\Z_{k_B}] + \tr[E_8(\Id - \Z_{k_B})] \ge 0$. Thus, $(u_A-\underline{u}_A)(u_B-\underline{u}_B)\ge0$ is implied by the product inequality $(\tr[E_{1}\Z_{k_A}] + \tr[E_2(\Id - \Z_{k_A})] + \tr[E_{2}\Z_{k_B}] + \tr[E_4(\Id - \Z_{k_B})] ) (\tr[E_{5}\Z_{k_A}] + \tr[E_6(\Id - \Z_{k_A})] + \tr[E_{7}\Z_{k_B}] + \tr[E_8(\Id - \Z_{k_B})] ) \ge 0$. Using rules $\tr[D_1 \Z_1] \tr[D_2 \Z_2] = \tr[(D_1 \mtens D_2)(\Z_1 \mtens \Z_2)]$, we can rewrite the  product inequality as the sum of trace inequalities applying suitable projections against tensor McCormick inequalities in \eqref{eq.psd}.
\end{proof}

Note that scalar bounds implied by $\Id\succeq\Z\succeq0$ do not yield RLT inequalities stronger than those produced by the tensor RLT.  Therefore, the convex outer approximation $\cC$ is defined by constraints in the DDPS system \eqref{eq.rmmr} and the scalar and tensor McCormick inequalities.

We now present the tensor RLT procedure on multiple SDP constraints, thus fully generalizing the Sherali-Adams RLT for scalar polynomials \cite{sherali2007rlt}. Consider the alphabet of matrix symbols
\[
\fZ \deq \{\Z_k,\Id_k\}_{k\in N},
\]
and its extension
\[
\bar{\fZ}\deq\{\Z_k,\Z_k^\top,\Id_k,\tr(\Z_k),\tr(\Z_k^\top)\}_{k\in N},
\]
where $\Z_k^\top$ denotes the transpose. Let $\langle\fZ\rangle$ be the family of finite ordered multisets over $\fZ$. Each multiset $\fS\in\langle\fZ\rangle$ gives rise to a tensor monomial $\bmtens_{\Z\in\fS}\Z$, and the monomials are in general noncommutative.

We extend the partial trace and partial transpose maps to act on tensor monomials. For $M\in\bmtens_{k\in[m]}\C^{d_k\times d_k}$ and a subset $S\subseteq[m]$, the partial trace $\tr_S$ and partial transpose $\Tr_S$ are defined entry-wise by the usual index contractions and index swaps, respectively:
\begin{align}
\label{eq.trace}
\tr_S\bigl([M_{i_1,j_1,\dots,i_m,j_m}]\bigr)
&= [\sum_{i_k,j_k:k\in S} M_{i_1,j_1,\dots,i_m,j_m}],\\
\label{eq.transpose}
\Tr_S\bigl([M_{i_1,j_1,\dots,i_m,j_m}]\bigr)
&= [M_{\bar i_1,\bar j_1,\dots,\bar i_m,\bar j_m}],
\end{align}
where $\bar{i}_k = \begin{cases}
    i_k\quad k \notin S\\
    j_k\quad k \in S
\end{cases}$, and $\bar{j}_k = \begin{cases}
    j_k\quad k \notin S\\
    i_k\quad k \in S
\end{cases}$ (indices are swapped on the subsystems in $S$).

The following rewriting lemma is useful.

\begin{theorem}
\label{thm.rewrite}
Any tensor monomial $\bar{\fS}\in\langle\bar{\fZ}\rangle$ can be written as $\tr_{S_A}\!\bigl(\Tr_{S_B}(\fS)\bigr)$ for some $\fS\in\langle\fZ\rangle$ and suitable index sets $S_A,S_B$.
\end{theorem}
\begin{proof}
Replace every occurrence of a trace symbol $\tr(\cdot)$ by an explicit partial trace over the corresponding subsystem (this produces $\tr_{S_A}$). Then, replace every transpose symbol by a corresponding partial transpose $\Tr_{S_B}$. The remaining object is a tensor monomial in the original alphabet $\fZ$; thus the original monomial equals $\tr_{S_A}\bigl(\Tr_{S_B}(\fS)\bigr)$ for some $\fS\in\langle\fZ\rangle$.
\end{proof}

The tensor RLT constructs valid SDP inequalities by the following steps, applied to a finite multiset $\cF'$ of factor constraints drawn from
\[
\cF=\{\Z,\Id-\Z,\Z^\top,\Id-\Z^\top,\tr(\Z)-1,\tr(\Z^\top)-1:\Z\in\fZ\}.
\]

\begin{enumerate}
    \item (Constraint-factor) Form the tensor product $\bmtens_{f\in\cF'} f$, which is PSD (or zero) by construction.
    \item (Product expansion) Expand the tensor product as a linear combination of monomials in $\bar{\fZ}$:
    \[
      \bmtens_{f\in\cF'} f \;=\; \sum_{i=1}^\ell a_i\,\bar{\fS}_i.
    \]
    \item (Monomial rewriting) Apply Theorem \ref{thm.rewrite} to write each $\bar{\fS}_i$ as $\tr_{S_{A_i}}\bigl(\Tr_{S_{B_i}}(\fS_i)\bigr)$ with $\fS_i\in\langle\fZ\rangle$.
    \item (Linearization) Introduce an auxiliary matrix variable $X_{\fS_i}$ for each monomial $\fS_i$ and replace $\fS_i$ by $X_{\fS_i}$. This yields SDP constraints of the form
    \[
      \sum_{i=1}^\ell a_i\,\tr_{S_{A_i}}\!\bigl(\Tr_{S_{B_i}}(X_{\fS_i})\bigr)\succeq(\,=)\ 0,\qquad X_{\fS_i}\succeq0.
    \]
\end{enumerate}

These matrix inequalities are tensor RLT inequalities. The procedure generates many useful constraints; we illustrate it with a short example.

\begin{example}
Take $\cF'=(\Z_1,\Id-\Z_2,\tr(\Z_3^\top)-1)$. Expanding yields
\[
\Z_1\mtens\Id\mtens\tr(\Z_3^\top)-\Z_1\mtens\Z_2\mtens\tr(\Z_3^\top)-\Z_1\mtens\Id+\Z_1\mtens\Z_2.
\]
Rewriting and linearizing produce, for instance,
\[
\tr_{\{3\}}\Tr_{\{3\}}(X_{\Z_1,\Id,\Z_3})
-\tr_{\{3\}}\Tr_{\{3\}}(X_{\Z_1,\Z_2,\Z_3})
- X_{\Z_1,\Id} + X_{\Z_1,\Z_2} = 0,
\]
a valid tensor RLT equality in the lifted variables $X_{\cdot}$.
\end{example}

Finally, many DPS constraints can be recovered from tensor RLT constructions. For instance:
\begin{itemize}
    \item[(a)] \textit{Positivity}: Let $\cF'=(\rho_A) \cup (\rho_B)_{k \in [\ell]}$, then the corresponding RLT inequality is $\rho_{AB_{[\ell]}} \succeq 0$.
    \item[(b)] \textit{Partial Trace Projection}: Let $\cF'=(\rho_A,\rho_B) \cup (\tr(\rho_B) - 1)_{k \in [2:\ell]}$, then the corresponding RLT equality is $   \tr_{B_{[2:\ell]}} (\rho_{AB_{[\ell]}}) - \rho_{AB} = 0$.

    \item[(c)] \textit{PPT}: Let $\cF'=(\rho_A) \cup (\rho^\top_B)_{k \in [s]} \cup (\rho_B)_{k \in [s+1:\ell]}$, then the corresponding RLT inequality is $ \Tr_{B_{[s]}}(\rho_{AB_{[\ell]}}) \succeq 0$.
\end{itemize}

Hence the tensor RLT provides a unified framework that subsumes DPS-type constraints and generates additional structured SDP inequalities useful for relaxations of multipartite separability.

\subsection{Branching}
We now describe the branching rule used in our sBB algorithm.
At an open node of the sBB search tree, the branching rule selects an entry of some intermediate variable $\Z_k$ and tightens its bounds when creating the subnodes. A good branching choice should tighten the relaxations in the subnodes so that the parent node's relaxation solution becomes infeasible in at least one subnode; otherwise the subnodes will inherit the same lower bound and the tree search will not make progress. Such a branching rule is called violating, and it is important for the sBB algorithm to make consistent progress.

We propose a violating branching rule that targets the relaxed equation $\Z_k = \Z_{k_A}\mtens\Z_{k_B}$. Let the values of $\Z_k,\Z_{k_A},\Z_{k_B}$ in the parent node's relaxation solution be the reference values. The rule evaluates a violation score for each entry and then selects the entry with the highest score as the branching candidate. For instance, suppose we branch on the real entry $u_A = U_{k_A,i_A j_A}$ of $\Z_{k_A}$. Given its current bounds $[\underline{u_A},\overline{u_A}]$ and its reference value $u'_A\in[\underline{u_A},\overline{u_A}]$, we create two subnodes with bounds $[\underline{u_A},u'_A]$ and $[u'_A,\overline{u_A}]$. The McCormick inequalities \eqref{eq.mccons} are updated in each subnode using the new bounds.

The proposed rule is similar in spirit to strong branching for mixed-integer programs: we simulate splitting at the reference value for each candidate entry and estimate how much the reference values are violated by the updated McCormick inequalities \eqref{eq.mccons}. Concretely, consider the scalar variables that appear in a single entry relation \eqref{eq.zk}:
the real and imaginary parts $u_A,v_A,u_B,v_B$ and the linearized variables $\mu_{AB},\nu_{AB}$, with reference values $u'_A,v'_A,u'_B,v'_B,\mu'_{AB},\nu'_{AB}$. Splitting at $u'_A$ produces two subnodes; for each subnode we compute the amount by which the reference point violates the McCormick constraints. The per-constraint violations are
\begin{subequations}
\label{eq.mcconsviolate}
\begin{align}
\max\Bigl(\max_{a\in\underline{u_A u_B},\,a'\in\overline{v_A v_B}}
\bigl(a(u'_A,u'_B)-a'(v'_A,v'_B)\bigr)-\mu'_{AB},\,0\Bigr),\\[2pt]
\max\Bigl(\mu'_{AB}-\max_{a\in\underline{u_A u_B},\,a'\in\overline{v_A v_B}}
\bigl(a(u'_A,u'_B)-a'(v'_A,v'_B)\bigr),\,0\Bigr),\\[2pt]
\max\Bigl(\max_{a\in\underline{u_A v_B},\,a'\in\underline{v_A u_B}}
\bigl(a(u'_A,v'_B)+a'(v'_A,u'_B)\bigr)-\nu'_{AB},\,0\Bigr),\\[2pt]
\max\Bigl(\nu'_{AB}-\max_{a\in\overline{u_A v_B},\,a'\in\overline{v_A u_B}}
\bigl(a(u'_A,v'_B)+a'(v'_A,u'_B)\bigr),\,0\Bigr).
\end{align}
\end{subequations}
Summing these four terms yields the total violation for a given subnode. Let the sums for the two subnodes be $s_1$ and $s_2$, respectively.

\begin{proposition}
    If $s_1>0$ and $s_2>0$, then branching on $u_A$ renders the reference point infeasible in both subnodes.
\end{proposition}
\begin{proof}
     By construction, $s_i>0$ implies that at least one McCormick inequality from \eqref{eq.mccons} is violated at the reference point under the subnode bounds; hence, the reference point is infeasible in that subnode. If both $s_1>0$ and $s_2>0$, infeasibility holds in both subnodes.
\end{proof}

We combine the two subnode scores into a single violation score to prioritize branching candidates. Empirically we use
\[
s \;=\; \tfrac{5}{6}\min(s_1,s_2) \;+\; \tfrac{1}{6}\max(s_1,s_2),
\]
 so the violating branching rule selects the variable with the largest $s$.

This violating rule applies generally to bilinear programs where McCormick relaxations are used. Other branching strategies exist for complex variables and PSD constraints (e.g., rules that force the rank of a PSD matrix to one) \cite{chen2017spatial}, but those typically estimate objective change and are not directly tailored to bilinear McCormick structures as used here.

\subsection{Heuristics}

The sBB algorithm also needs heuristics for finding solutions that provide upper bounds for the BSS problem; we use a simple alternating heuristic.

The alternating heuristic iteratively updates rank-1 density matrices $\{\X_k\}_{k \in [m]}$ that provide a solution $\y = \bmtens_{k \in [m]} \X_k$ for \eqref{eq.sep}. In each iteration, the heuristic fixes $m-1$ density matrices and maximizes the objective $\inner{\chi}{\bmtens_{k \in [m]}\X_k}$ with respect to the remaining free density matrix. Given $\{\X'_k\}_{k \in [m]}$ from the previous iteration and assuming the index $k'$ is free, the restricted BSS problem is equivalent to the following SDP problem:
\begin{equation}
     \max_{\forall k \in [m]/ \{k'\}, \X_k = \X'_k, \X_{k'}  \in \dmat^{d_{k'}}} \inner{\chi}{\bmtens_{k \in [m]}\X_k} =  \max_{\X_{k'}  \in \dmat^{d_{k'}}} \inner{\chi_{k'}}{\X_{k'}},
\end{equation}
where $\chi_{k'}$ is a tensor contraction of $\chi$ on the $k'$-th mode.
Note that the optimal solution $\X_{k'}$ must be a rank-1 matrix $\x  \mtens \x^\dagger$, where $\x$ is the unit eigenvector corresponding to the most negative eigenvalue of the matrix $\chi_{k'}$.

We initialize the alternating heuristic with solutions of the convex relaxation at the search-tree nodes. We choose the free index $k'$ cyclically and stop the alternating heuristic when it converges or when it reaches the iteration limit. This heuristic was first introduced in \cite{Shang2018convex} for separability certification and applied in \cite{liu2025unified} as well.

\section{Numerical experiments}
\label{sec.numerical}
In this section, we present the numerical experiments of the existing and proposed algorithms on the white-noise mixing
threshold problem. The data, detailed results, and source code are available on the online repository: \url{https://github.com/ZIB-IOL/TensorOpt4Entanglement}.

\subsection{Development environment.}
All experiments were conducted on a cluster equipped with Intel Xeon Gold 5122 CPUs, operating at 3.60 GHz with 96 GB of RAM. Each test was executed on a single thread with a memory limit of 10GB RAM. All algorithms are implemented in the Julia language and JuMP modelling language. The following solvers and software packages were used: Mosek (for solving LPs and SDPs), Manopt \cite{bergmann2022manopt} (for using quasi-Newton method), and Ket.jl \cite{araujo_2025_15837771} (for generating quantum states and computing lower bounds on white-noise mixing
threshold).

\subsection{Problem instances.}
 Our benchmark consists of the following instances:
\begin{itemize}
    \item Greenberger–Horne–Zeilinger (GHZ) states \cite{greenberger1990bell}: GHZ states are typical examples of maximally entangled multipartite  states. Its white-noise threshold is analytically known as $1 - 1 / (1 + 2^{m-1})$ for $m\ge 2$ \cite{Dur2000classification}.
    \item Dicke states \cite{Dicke1954coherence}: Dicke states describe the highly entangled state of an ensemble with fixed energy excitation.
    Their entanglement is maximally persistent and robust under particle losses \cite{Guhne2008multiparticle}.
        \item Cluster states  \cite{Briegel2001persistent}: Cluster states describe a type of entangled states in lattices with Ising-type interactions. Their entanglement is more robust to noise and local operations. For $m=2$ and $m=3$, the cluster states are equivalent to the Bell state and GHZ state (up to local transformations), respectively.
\end{itemize}
Note that all the instances are states of $m$-partite systems, where the dimension of each subsystem is $2 \times 2$.

\subsection{Algorithms implementations and details}
\label{sec.implem}

We implement the following algorithms:

\begin{itemize}
    \item SDP-based alternating optimization (Alt-SDP): our implementation follows \cite{Ohst2024certifying}. The method represents separable states as a sum of $r$ tensor components, each a Kronecker product of subsystem density matrices (see \Cref{prop.state}). In each iteration, one density matrix per component is left free and updated, so each update reduces to an SDP and yields an upper bound on the white-noise mixing threshold. Our implementation sets $r$ the same as the factorization size in the other algorithms (the same baseline).
    \item Primal-dual geometric reconstruction (PDGR) method: We also report the best available upper and lower bounds computed by PDGR \cite{liu2025unified}, which is backed by the advanced CG algorithms implemented in FrankWolfe.jl \cite{besanccon2022frankwolfe}. In the trial step, the method applies a line search with the CG-based entanglement detection to find entangled and interpolated states $\rho(\z; \phi,\Id)$ for $ 1 \ge \z \ge 0$ (defined in \eqref{eq.thre}). The corresponding  $\z$ establishes lower bounds on the white-noise mixing threshold. During each line-search iteration, the PDGR method also tries to reconstruct a separable state $\rho(\z^*; \phi,\Id)$ between $\rho(\z; \phi,\Id)$ and $\Id/\bar{d}$ using the separability certification of GR. The corresponding $\z^*$ are upper bounds on the white-noise mixing threshold.

    \item LADMM (\Cref{algo.alm}): the quasi-Newton solver from Manopt is used for solving lifted nonconvex subproblems. LADMM typically returns a heuristic solution of \eqref{eq.thre} and therefore a heuristic upper bound together with a feasibility residual $\| \y - A(\z) - \a\|_2$. Following \Cref{rmk.crossover}, we optionally run one CP iteration to refine the LADMM output and obtain a nearby feasible point with a certified upper bound.
    \item Cutting-plane (CP \Cref{algo.dual}): see \Cref{exp.ref3} for the LPs solved inside the CP routine. The CP algorithm uses the tailored sBB method as the LMO for the BSS problem and provides both upper and lower bounds on the white-noise mixing threshold.
    \item Iterative refinement (IR \Cref{algo.lift}): the IR alternates between LADMM and CP to produce tightened upper and (Lagrangian) lower bounds on the white-noise mixing threshold.
    \item DPS: we use  Ket.jl to obtain lower bounds provided by \eqref{eq:PR} with the DPS hierarchy (see \Cref{rmk.dps}).
    \item DDPS+: we construct the relaxation \eqref{eq:PR} by the proposed (first-level) DDPS outer approximation \eqref{eq.dmmr}  with  scalar  McCormick inequalities  \eqref{eq.mccons} and tensor McCormick inequalities \eqref{eq.psd}, which provides lower bounds.
\end{itemize}

Experimental settings and parameters are as follows.

\begin{itemize}
    \item Time limits and tolerances: time limits are set to 1/2/3 hours for $m=3/4/5$, respectively. Numerical tolerances are $10^{-6}$ unless stated otherwise.
    \item Factorization size and initialization: the factorization size $r$ is set to 256/512/652 for $m=3/4/5$, respectively. We generate $r$ random rank-1 tensors of the form $\x_{1}\mtens\x_{1}^\dagger\cdots\mtens\x_{m}\mtens\x_{m}^\dagger$ (properly normalized) to initialize the algorithms that use factorizations.
    \item CP and sBB LMO: within the IR algorithm, the CP subroutine iteration limit is set to the factorization size; the standalone CP routine has no iteration limit and terminates when upper and lower bounds coincide. For the sBB LMO we set the node limit to 120/100/1 for $m=3/4/5$, respectively. Before 90\% of the global time limit is reached, the sBB stops as soon as a violated cut is found; if no solution is found within the node limit, we fall back to the alternating heuristic with several random starts. After 90\% of the time limit, the CP routine focuses on improving dual lower bounds by running the sBB up to a tighter node limit of 70/50/15 for $m=3/4/5$ (gap-closing CP iterations), respectively.
    \item LADMM settings: in the IR algorithm, the LADMM iteration limits are 10/10/4 for $m=3/4/5$, respectively; for the standalone LADMM runs we allow unlimited iterations. The quasi-Newton solver is configured with a maximum of 150 iterations, objective tolerance $10^{-5}$, and step tolerance $10^{-4}$. The initial penalty parameter $\zeta$ is set to 1/5/10 for $m=3/4/5$, respectively. We adapt $\zeta$ using the scheme suggested in \cite{wang2025solving}:
    \begin{equation}
        \zeta = \begin{cases}
            2.5 \zeta  & \text{if } \norm{\y - A(\z) - \a}_2 > 0.8 \norm{\frac{\partial L}{\partial \x}(\z,\Psi(\x),\chi)}_2 \\
            0.4 \zeta  & \text{if } \norm{\y - A(\z) - \a}_2     \geq 0.8 \norm{\frac{\partial L}{\partial \x}(\z,\Psi(\x),\chi)}_2.
        \end{cases}
    \end{equation}
    The LADMM terminates if the feasibility residual and the norm of the gradient is below the tolerance, i.e., $\| \y - A(\z) - \a\|_2 < 10^{-6}$ and $ \norm{\frac{\partial L}{\partial \x}(\z,\Psi(\x),\chi)}_2 < 10^{-6}$.
    \item DPS parameters: we set the number of extended copies as large as allowed by  Ket.jl. Concretely, for $m=3/4/5$, we set the DPS configuration $(d_{A'},d_{B'},\ell)$ in \Cref{rmk.dps} to (4,2,7)/(4,4,3)/(8,4,2), respectively. These DPS-based relaxations provide lower bounds on the white-noise mixing threshold.
\end{itemize}

 \subsection{Results and analysis}

The experimental results are shown in \Cref{tab.m3,tab.m4,tab.m5} for states of $m=3,4,5$ subsystems, respectively. The columns show the upper bound and the lower bound that are considered numerically reliable, the heuristic upper bound, and the feasibility residual of the heuristic solution given by the algorithms. The final column shows the total runtime in seconds. A dash indicates that the metric corresponding to the column and row is not available. Boldface indicates the best upper and lower bounds for an instance. The PDGR method yields the best lower bounds for  many instances of $m=3,4$. Excluding this method, our DDPS+ relaxation consistently attains the strongest lower bounds.
 IR finds the best upper bounds for all instances. Moreover, for instances of $m=3$, the LADMM algorithm has found solutions with feasibility residual less than $10^{-5}$, but the crossover method in \Cref{rmk.crossover} does not lead to a non-trivial upper bound (a failed crossover).
\begin{table}[htb]
\centering
\begin{tabular}{l|l|ccccc}
\toprule
    \textbf{State} & \textbf{Algorithm} & $\ub_{\relx}$ & $\lb_{\relx}$ & $\ub_{\heur}$ & $\feas_{\heur}$ & \textbf{Time (s)} \\
\toprule
 \multirow{4}{*}{GHZ\_3}
& Alt-SDP & 0.80070 & - & - & - & 48 \\
& PDGR & 0.80002  & 0.79659 & - & - & - \\
 & LADMM & 1.00000 & - & 0.80000 & 0.00001 & 3606 \\
 & CP & 0.80007 & 0.79584 & - & - & 3602 \\
 & IR & \textbf{0.80000} & \textbf{0.80000} & 0.80011 & 0.00072 & 289 \\
 & DPS & - & 0.06250 & - & - & 3571 \\
  & DDPS+ & - & \textbf{0.80000} & - & - & 15 \\

\midrule
 \multirow{4}{*}{Dicke\_3\_1}
& Alt-SDP & 0.82263 & - & - & - & 29 \\
& PDGR &  0.82208 & \textbf{0.81856} & - & - & - \\
 & LADMM & 1.00000 & - & 0.82203 & 0.00009 & 3606 \\
 & CP & 0.82268 & 0.61398 & - & - & 3611 \\
 & IR & \textbf{0.82203} & 0.61661 & 0.82180 & 0.00143 & 3604 \\
 & DPS & - & 0.05893 & - & - & 4507 \\
 & DDPS+ & - & 0.79041 & - & - & 15 \\

\midrule
 \multirow{4}{*}{Dicke\_3\_2}
& Alt-SDP & 0.82244 & - & - & - & 37 \\
 & LADMM & 1.00000 & - & 0.82202 & 0.00007 & 3606 \\
 & CP & 0.82294 & 0.61283 & - & - & 3603 \\
 & IR & \textbf{0.82203} & 0.61624 & 0.82209 & 0.00172 & 3610 \\
 & DPS & - & 0.05893 & - & - & 4512 \\
  & DDPS+ & - & \textbf{0.79041} & - & - & 15 \\
\bottomrule

\end{tabular}
\caption{Experimental results for $m=3$.}
\label{tab.m3}
\end{table}

\begin{table}[htb]
\centering
\begin{tabular}{l|l|ccccc}
        \textbf{State} & \textbf{Algorithm} & $\ub_{\relx}$ & $\lb_{\relx}$ & $\ub_{\heur}$ & $\feas_{\heur}$ & \textbf{Time (s)} \\
\toprule
 \multirow{4}{*}{Dicke\_4\_2}
& Alt-SDP & \textbf{0.91430} & - & - & - & 266 \\
& PDGR & 0.91439  & \textbf{0.90211} & - & - & 266 \\
 & LADMM & 1.00000 & - & 0.91433 & 0.00008 & 7206 \\
 & CP & 0.91442 & 0.56803 & - & - & 7205 \\
 & IR & \textbf{0.91430} & 0.57071 & 0.91492 & 0.00203 & 7223 \\
 & DPS & - & 0.02083 & - & - & 3084 \\
& DDPS+ & - & 0.74985 & - & - & 15 \\

\midrule
 \multirow{4}{*}{Cluster\_4}
& Alt-SDP & 0.89635 & - & - & - & 326 \\
 & LADMM & 0.88889 & - & 0.88889 & 0.00001 & 7206 \\
 & CP & 0.90164 & 0.60550 & - & - & 7220 \\
 & IR & \textbf{0.88890} & 0.62595 & 0.92246 & 0.28444 & 7220 \\
 & DPS & - & 0.03125 & - & - & 3433 \\
  & DDPS+ & - & \textbf{0.72727} & - & - & 15 \\

\midrule
 \multirow{4}{*}{GHZ\_4}
& Alt-SDP & 0.88933 & - & - & - & 1049 \\
& PDGR & 0.88901 &  \textbf{0.87751} & - & - & - \\
 & LADMM & 1.00000 & - & 0.88889 & 0.00005 & 7206 \\
 & CP & 0.89138 & 0.59891 & - & - & 7205 \\
 & IR & \textbf{0.88890} & 0.57630 & 0.94108 & 0.96656 & 7222 \\
 & DPS & - & 0.03125 & - & - & 3411 \\
 & DDPS+ & - & 0.72727 & - & - & 15 \\

\midrule
 \multirow{4}{*}{Dicke\_4\_1}
& Alt-SDP & 0.90889 & - & - & - & 533 \\
& PDGR &  0.90757 & \textbf{0.89450} & - & - & - \\
 & LADMM & 1.00000 & - & 0.91158 & 0.05048 & 7232 \\
 & CP & 0.90779 & 0.64677 & - & - & 7234 \\
 & IR & \textbf{0.90748} & 0.65802 & 0.69942 & 0.59578 & 7240 \\
 & DPS & - & 0.03125 & - & - & 3392 \\
 & DDPS+ & - &  0.72727 & - & - & 15 \\
\bottomrule

\end{tabular}
\caption{Experimental results for $m=4$.}
\label{tab.m4}
\end{table}

\begin{table}[htb]
\centering
\begin{tabular}{l|l|ccccc}
\toprule
    \textbf{State} & \textbf{Algorithm} & $\ub_{\relx}$ & $\lb_{\relx}$ & $\ub_{\heur}$ & $\feas_{\heur}$ & \textbf{Time (s)} \\
\toprule
 \multirow{4}{*}{Cluster\_5}
& Alt-SDP & 0.94712 & - & - & - & 10852 \\
 & LADMM & 1.00000 & - & 0.96488 & 0.04053 & 10935 \\
 & CP & 0.95594 & 0.58307 & - & - & 10948 \\
 & IR & \textbf{0.94334} & 0.41951 & 0.94054 & 0.00303 & 10820 \\
 & DPS & - & 0.01562 & - & - & 160 \\
 & DDPS+ & - & \textbf{0.84211} & - & - & 18 \\

\midrule
 \multirow{4}{*}{Dicke\_5\_1}
& Alt-SDP & 0.95520 & - & - & - & 7623 \\
 & LADMM & 1.00000 & - & 0.97327 & 0.00957 & 11887 \\
 & CP & 0.95685 & 0.60589 & - & - & 10811 \\
 & IR & \textbf{0.95403} & 0.59546 & 0.94929 & 0.00454 & 10920 \\
 & DPS & - & 0.01531 & - & - & 160 \\
 & DDPS+ & - & \textbf{0.83937} & - & - & 18 \\

\midrule
 \multirow{4}{*}{Dicke\_5\_2}
& Alt-SDP & 0.95962 & - & - & - & 7585 \\
 & LADMM & 1.00000 & - & 0.97267 & 0.02212 & 11646 \\
 & CP & 0.96430 & 0.56383 & - & - & 10979 \\
 & IR & \textbf{0.95894} & 0.51509 & 0.95099 & 0.00764 & 10973 \\
 & DPS & - & 0.01326 & - & - & 157 \\
 & DDPS+ & - & \textbf{0.86014} & - & - & 19 \\

\midrule
 \multirow{4}{*}{GHZ\_5}
& Alt-SDP & 0.94327 & - & - & - & 10824 \\
 & LADMM & 1.00000 & - & 0.96762 & 0.28485 & 10839 \\
 & CP & 0.95829 & 0.55766 & - & - & 10895 \\
 & IR & \textbf{0.94163} & 0.37147 & 0.92536 & 0.02518 & 10855 \\
 & DPS & - & 0.01562 & - & - & 155 \\
  & DDPS+ & - &\textbf{0.84211} & - & - & 18 \\
\bottomrule

\end{tabular}
\caption{Experimental results for $m=5$.}
\label{tab.m5}
\end{table}

\subsubsection{Detailed analysis}
The experimental tables show consistent patterns across the three problem sizes.

The upper bound results given by the PDGR method are close to the best results given by the IR method for $m=3,4$. However, the  lower bounds by the PDGR method are better, and the gaps of the PDGR method are also smaller. The simplicial partition can cover the separable states with an approximation error less than 1\% with a payoff of an exponential number (in $m$) of simplices, so the associated LMO can return almost optimal solutions to low-dimensional BSS problems. However, due to the extremely large size of the simplicial partition, the PDGR method can not scale up for large instances and fails to provide the results for $m=5$.

The bipartite DPS lower bounds are cheap to obtain but quite weak compared to those obtained by the DDPS+. In addition, the Lagrangian bound given by CP is weaker than the relaxation bound given by DDPS+.

For $m=3$, the Alt-SDP heuristic quickly produces reasonable upper bounds (on the other hand, it tends to converge rapidly but may stagnate at local minima); LADMM produces good heuristic objectives but sometimes violates the constraint with feasibility residuals below $10^{-5}$. Nevertheless, the crossover procedure described in \Cref{rmk.crossover} does not always yield a certified feasible solution  (notably  $\ub_{\relx}=1.0$ returned after crossover). In practice, this failure is partly attributable to the LP solver's tight feasibility tolerance (Mosek uses $10^{-8}$), so a single CP iteration can be insufficient to convert an almost-feasible solution into a provably feasible one.  While the IR algorithm refines LADMM outputs.

For $m=4$, runtimes increase substantially: Alt-SDP still gives fast decent upper bounds, LADMM continues to improve objectives but typically requires CP to obtain feasibility and numerically reliable bounds; the CP stage becomes more time-consuming as SDP relaxations grow.  This justifies the need for a full CP algorithm after the IR algorithm is run.

\begin{table}[htb]
\centering
\begin{tabular}{l|l|ccccc}
\toprule
    \textbf{State} & \textbf{Algorithm} & $\ub_{\relx}$ & $\lb_{\relx}$ & $\ub_{\heur}$ & $\feas_{\heur}$ & \textbf{Time (s)} \\
\toprule

 \multirow{4}{*}{Cluster\_5}
& LADMM\_400 & 1.00000 & - & 0.95027 & 0.04640 & 10854 \\
 & LADMM\_500 & 1.00000 & - & 0.94483 & 0.01052 & 10816 \\
 & LADMM\_600 & 1.00000 & - & 0.94818 & 0.01407 & 11637 \\
 & LADMM\_700 & 1.00000 & - & 0.96047 & 0.04056 & 11025 \\
 & LADMM\_800 & 1.00000 & - & 0.93631 & 0.01491 & 10832 \\
 & LADMM\_900 & 1.00000 & - & 0.96534 & 0.24192 & 10918 \\

\midrule
 \multirow{4}{*}{Dicke\_5\_1}
& LADMM\_400 & 1.00000 & - & 1.00000 & 0.98425 & 10834 \\
 & LADMM\_500 & 1.00000 & - & 0.95322 & 0.00674 & 10873 \\
 & LADMM\_600 & 1.00000 & - & 0.98213 & 0.02546 & 11546 \\
 & LADMM\_700 & 1.00000 & - & 0.95160 & 0.01705 & 11155 \\
 & LADMM\_800 & 1.00000 & - & 0.91595 & 0.03606 & 10863 \\
 & LADMM\_900 & 1.00000 & - & 0.87148 & 0.08047 & 11453 \\

\midrule
 \multirow{4}{*}{Dicke\_5\_2}
& LADMM\_400 & 1.00000 & - & 0.95992 & 0.01926 & 10828 \\
 & LADMM\_500 & 1.00000 & - & 0.96941 & 0.04906 & 10845 \\
 & LADMM\_600 & 1.00000 & - & 0.96265 & 0.05245 & 11191 \\
 & LADMM\_700 & 1.00000 & - & 0.96002 & 0.01333 & 11273 \\
 & LADMM\_800 & 1.00000 & - & 0.95767 & 0.00023 & 10861 \\
 & LADMM\_900 & 1.00000 & - & 0.95767 & 0.00055 & 11160 \\

\midrule
 \multirow{4}{*}{GHZ\_5}
& LADMM\_400 & 1.00000 & - & 0.94117 & 0.00003 & 10812 \\
 & LADMM\_500 & 1.00000 & - & 0.94116 & 0.00003 & 10812 \\
 & LADMM\_600 & 1.00000 & - & 0.94119 & 0.00001 & 10813 \\
 & LADMM\_700 & 1.00000 & - & 0.94117 & 0.00109 & 10940 \\
 & LADMM\_800 & 1.00000 & - & 0.94117 & 0.00012 & 10840 \\
 & LADMM\_900 & 1.00000 & - & 0.94115 & 0.00004 & 10813 \\
\bottomrule

\end{tabular}
\caption{Results of low-rank approximations for LADMM for $m=5$.}
\label{tab.m5low}
\end{table}

\subsubsection{Analysis of low-rank approximations}

\Cref{prop.compact} gives a trivial bound on the compact factorization size (one may take \(r\le 2D+1\) when \(\dim(\SOST_{\F}(\vd))=D\)). For $m=3/4$, we already set the factorization size compact. For $m=5$, we use a low-rank approximation for these large-scale instances; \Cref{tab.m5low} summarizes the results of the LADMM with factorization size $r \in \{400, 500, 600, 700, 800, 900\}$.
Empirically, a larger factorization size $r$ tends to produce better heuristic objectives from LADMM, but the improvement is not strictly monotone. Different parameterizations can lead the nonconvex solver to different local minima, so objective and feasibility behaviour may vary with $r$. Feasibility residuals also vary: some LADMM outputs are nearly feasible while others require substantial post‑processing. Some states (e.g., GHZ\_5) are robust to $r$, while others (e.g., Dicke\_5\_1, Cluster\_5) are sensitive. These observations suggest  that low-rank assumptions in general do not hold (except for GHZ states).
Regardless of the chosen \(r\), the IR pipeline remains necessary: high-quality warm starts from LADMM should be converted and certified via CP (using the sBB LMO) to obtain feasible, provably valid upper and lower bounds. This balanced, resource-aware method yields the best trade-off between solution quality and computational cost.

\begin{table}[htb]
\centering
\begin{tabular}{l|lccc}
\toprule
    \textbf{State} & \textbf{Algorithm} & $\ub_{\relx}$ & $\lb_{\relx}$ & $\underbar{b}$   \\
\midrule
\multirow{2}{*}{Cluster\_5}
 & CP & 0.95600 & 0.57313 & -0.38429 \\
 & IR & 0.94336 & 0.41713 & -0.53011 \\
\midrule
\multirow{2}{*}{Dicke\_5\_1}
 & CP & 0.95690 & 0.59113 & -0.35414 \\
 & IR & 0.95405 & 0.59113 & -0.36641 \\
\midrule
\multirow{2}{*}{Dicke\_5\_2}
 & CP & 0.96435 & 0.56113 & -0.40058 \\
 & IR & 0.95895 & 0.51509 & -0.44928 \\
\midrule
\multirow{2}{*}{GHZ\_5}
 & CP & 0.95835 & 0.55721 & -0.40034 \\
 & IR & 0.94164 & 0.43013 & -0.58542 \\
\bottomrule

\end{tabular}
\caption{Average results on gap-closing CP iterations for $m=5$.}
\label{tab.m5CP}
\end{table}

\subsubsection{Analysis of hard instances}

For $m=5$, the scalability limitations become apparent: we must restrict the factorization size and rely on low-rank approximations. The IR algorithm remains the most effective at decreasing the upper bound (particularly relative to Alt-SDP), but it is less effective at improving the (Lagrangian) lower bound than the pure CP procedure. Recall that the Lagrangian lower bound is computed as $\lb_{\relx}=\ub_{\relx}+\underbar{b}$, where $\underbar{b}$ is the sBB LMO's lower bound for the BSS separation subproblem. Averaged values of these quantities (the gap-closing CP iteration statistics: $\ub_{\relx}$, $\lb_{\relx}$, and $\underbar{b}$) are reported in \Cref{tab.m5CP}. Although IR yields a smaller $\ub_{\relx}$, its associated $\underbar{b}$ is substantially more negative than that obtained by the CP-only approach, leading to a worse (smaller) $\lb_{\relx}$.

This lower‑bound deterioration  highlights a shortcoming of the IR for large-scale problems: we reuse  LADMM's tensor factorization of $\x$ and the corresponding extreme components $R(\x)$ to warm-start the CP. However,  compared to the standalone CP, the CP subroutine in the IR is not run for a sufficient number of iterations, and the CP is still far from convergence. In our CP implementation, termination occurs when the LMO returns a nonnegative optimal value (no violated cuts). Otherwise, a violated cut (negative value) is added.  When the CP subroutine is not converging, the LMO's optimal solution can be unstable and the optimal value $b$ as well as its lower bound $\underbar{b}$ can be very negative. These observations suggest that, for large-scale problems, CP subroutines need many iterations to converge.

\subsubsection{Summary}

The Alt-SDP heuristic is easy to implement and computationally fast, but it provides no lower bounds and may miss small objective improvements that more expensive methods can capture. LADMM (lifted ADMM) often attains slightly better objectives than simple alternating updates and typically yields near-feasible solutions with small residuals (around $10^{-4}$). However, achieving high-accuracy convergence is extremely difficult (taking hours for larger instances). In addition, LADMM is sensitive to parameter tuning, can violate coupling constraints, and usually requires a CP crossover or refinement step to turn heuristic solutions into certified feasible solutions.

The PDGR method is more problem-specific, since it is designed for the line-search formulation. It can provide the best lower bounds and competitive upper bounds for small instances ($m < 5$). For larger-scale cases ($m = 5$), the CP/IR pipeline with the sBB LMO offers a reliable approach to numerically certified upper bounds and feasible solutions, while DDPS+ improves the SDP-based lower bounds produced by the bipartite DPS hierarchy. The principal computational bottleneck for these methods lies in the LMO: as $m$ increases, solution quality degrades and running times grow substantially. In addition, the general-purpose algorithms (CP/LADMM/IR) take more running time than the problem-specific algorithms (Alt-SDP/PDGR).

\section{Conclusion}

In this work, we study convex optimization over the cone of PSD tensors, a natural generalization of SDP with broad examples in quantum information theory. We leverage strategies inspired by modern SDP algorithms and incorporate new nonconvex structures arising in PSD tensors. Our work bridges the gaps between general-purpose algorithms in computational optimization and problem-specific algorithms in computational physics on a standard benchmark problem.

We develop the LADMM algorithm to tackle the lifted nonconvex formulation of SDTO. In addition, we design a CP algorithm that iteratively refines solutions of the LADMM algorithm, together with a tailored sBB method serving as the LMO. We then integrate these two approaches into an IR framework that alternates between primal and dual updates to generate upper and lower bounds on the optimal value.

Our numerical experiments on benchmark white-noise threshold instances show that the IR method consistently yields the best upper bounds and scales to larger problem sizes. However, SDTO problems remain substantially more challenging than standard SDPs, as the dimension grows exponentially with the number of subsystems (modes). The tailored PDGR method, for example, does not scale to these larger instances. This demonstrates the need for further research on more efficient algorithms and implementations for large-scale SDTO problems, particularly general-purpose approaches that can extend beyond the white-noise threshold setting.

Since many convex optimization problems in quantum information involving density matrices can be reformulated in terms of separable states, a broad class of such problems can be expressed as SDTOs. Thus, our framework is not restricted to the white-noise threshold problem. In particular, it can be extended to incorporate operator convex/concave functions acting on PSD tensors, such as quantum (relative) entropy \cite{hiai2014introduction}.

\subsection*{Acknowledgments}
This research was partially supported by Research Campus MODAL,
funded by the Federal Ministry of Research, Technology and Space (BMFTR)
(fund numbers 05M14ZAM, 05M20ZBM) and the DFG
Cluster of Excellence MATH+ (EXC-2046/1, Project No. 390685689) funded
by the Deutsche Forschungsgemeinschaft (DFG).

\ifthenelse {\boolean{mpc}}
{
\bibliographystyle{spmpsci}
}
{
\bibliographystyle{plainnat}
}
\bibliography{reference}

\end{document}